\documentclass[11pt]{article}
\usepackage[utf8]{inputenc}
\usepackage{amsfonts, amsmath}
\usepackage{mdframed}
\usepackage{amssymb}
\usepackage{amsthm,amsmath,amscd}
\usepackage{enumerate}
\usepackage{url}
\usepackage[pdftex]{graphicx}
\usepackage[margin=25pt,font=small,labelfont=bf]{caption}
\usepackage{subfig}
\usepackage[numbers,square,sort&compress]{natbib}
\usepackage[colorlinks=true]{hyperref}

\usepackage[
  colorinlistoftodos,
  backgroundcolor=yellow,
  textsize=footnotesize,
]{todonotes}

\usepackage[dvipsnames]{xcolor}

\newcommand{\defn}[1]{\textcolor{blue}{\emph{#1}}}
\newcommand*{\doi}[1]{doi:
\href{https://dx.doi.org/#1}{\urlstyle{rm}\nolinkurl{#1}}}
\newcommand*{\arxiv}[1]{arXiv:
\href{https://arxiv.org/abs/#1}{\urlstyle{rm}\nolinkurl{#1}}}
\let\oldproofname=\proofname
\renewcommand{\proofname}{\rm\bf{\oldproofname}}

\newcommand{\RR}{\mathbb R}

\newcommand{\R}{\mathbb R}

\newcommand{\bna}{
\begin{eqnarray}}
  \newcommand{\ena}{
\end{eqnarray}}
\newcommand{\ba}{
\begin{eqnarray*}}
  \newcommand{\ea}{
\end{eqnarray*}}
\newcommand{\bs}[1]{}

\newcommand{\f}{{\mathbf f}}

\newcommand{\eps}{\varepsilon}

\newtheorem{theorem}{Theorem}[section]
\newtheorem{corollary}[theorem]{Corollary}
\newtheorem{lemma}[theorem]{Lemma}
\newtheorem{proposition}[theorem]{Proposition}
\newtheorem{remark}[theorem]{Remark}
\newtheorem{definition}[theorem]{Definition}

\theoremstyle{definition}
\newtheorem{example}[theorem]{Example}

\def\p{{\bf p}}

\def\l{{\bf l}}

\def\m{{\bf m}}

\def\q{{\bf q}}
\def\0{{\bf 0}}
\def\b{{\bf b}}

\let\oldv=\v
\def\v{{\bf v}}

\def\w{{\bf w}}

\def\r{{\bf r}}
\def\u{{\bf u}}
\def\x{{\bf x}}

\def\y{{\bf y}}
\def\a{{\bf a}}

\newcommand{\me}[1]{^{(#1)}}

\usepackage[margin=1in]{geometry}

\begin{document}
\title{Higher Order Rigidity and Energy}

\author{
  Steven J. Gortler
  \and
  Miranda Holmes-Cerfon
  \and
  Louis Theran
}
\date{}
\maketitle

\begin{abstract}

  In this paper,
  we revisit the notion of higher-order rigidity of a bar-and-joint framework.
  In particular, we provide a link between the rigidity properties of
  a framework, and the growth order of an energy function defined on
  that framework.
  Using our approach,
  we propose a general definition for the rigidity order of a
  framework, and we show that this definition does not depend on the
  details of the chosen energy function.
  Then we show how this order can be studied using higher order
  derivative tests.
  Doing so, we obtain a new proof that the lack of a second order
  flex implies rigidity.
  Our proof relies on our construction of a fourth derivative test,
  which may be applied to a critical point when the second derivative
  test fails.
  We also obtain a new proof that when
  the dimension of non-trivial first-order flex coefficients
  $\p'$
  equals $1$, then
  the lack of a
  $k$th order flex for some $k$ implies a framework is rigid. The
  higher order derivative tests
  that we study here may have applications in more general optimization problems.

\end{abstract}

\section{Introduction}

One of the most basic problems in rigidity theory is determining if a
bar-and-joint framework is  rigid in $\RR^d$.
A bar-and-joint framework $(G,\p)$ consists of a graph $G$ and a configuration
$\p:=\{\p_1 ... \p_n\}$ of $n$ points in $\RR^d$. Each edge $ij$ of $G$
has a Euclidean length
$d_{ij}:=|\p_i-\p_j|$.
The framework is called (locally)
\defn{rigid} if there is a neighborhood $U$ of $\p$
so that
any $\q \in U$
with the same edge lengths (in $G$) as $\p$ must be
congruent to $\p$.
If a framework is not rigid, it is called \defn{flexible}.
A motion $\p(t)$ with $\p(0)=\p$ that preserves the edge lengths is called a
\defn{finite flex}.
A framework that is flexible must have an
analytic finite flex~\cite{Gluck}.

As presented, the question of determining whether a framework $(G,\p)$
is rigid involves solving the system of quadratic equations
$|\q_i - \q_j|^2=
d^2_{ij}$ over the variable configuration $\q$,
to determine if $\p$ is an isolated solution (once
congruences are removed). This  is an instance of an
intractable problem.
In particular, testing rigidity has been shown to
be co-NP hard~\cite{Abbot-Hard}, so there is likely no efficient
algorithm
that works in all cases.

Instead,
``one sided'' tests that either certify
rigidity or are inconclusive are used
in practice.
Two  most important tests are: first order rigidity (e.g.,
\cite{Asimow-Roth-I}),
and
second order rigidity \cite{connelly-second}.  These have the property that
\[
  \text{first order rigid} \Longrightarrow 
  \text{second order rigid} \Longrightarrow \text{rigid}.
\]
(There is also the important test of prestress
  stability~\cite{pss} that
  has a different flavor and that
we will address in an upcoming companion paper~\cite{HOpre}.)
First and second order
rigidity make use of the notion of a $k$th-order flex, which is
an analytic trajectory
$\p(t)$
that satisfies
the constraints of a finite flex
through $k$th order in $t$.  A framework is
called first order rigid if it does not
have a non-trivial first-order flex; this condition can be checked
with linear algebra.  A framework is
called second order rigid if it does not
have a non-trivial second-order flex; checking this is already a non-linear
problem for which no efficient algorithm is known.

One might expect that, if a framework is not second order rigid,
i.e., it has a non-trivial second-order flex, one could
simply define weaker notions of rigidity by
seeing if it lacks a
non-trivial $k$th-order flex, for
increasing values of $k\ge 3$.
Unfortunately, Connelly and Servatius \cite{conSer}
described a framework $(G,\p)$ that has no non-trivial third
order flex, but is in fact flexible. Any analytic finite flex
$\p(t)$ of the framework constructed in \cite{conSer} has  the property
that $\dot{\p}(t)|_{t=0}=0$, which means
the flex has to stop at $\p$ to ``turn around''.  A flexible framework
with this behavior is called a cusp mechanism.

There are several downstream consequences from the existence of a cusp
mechanism.  The first is that it explains, post hoc, why the proof
that second order rigidity implies rigidity is quite subtle.  It needs to
rule out, at least implicitly, the existence of a cusp mechanism
without a non-trivial
second order flex.  (No such complications can occur with first order
  rigidity, and
there are many ways to straightforwardly show it implies rigidity.)
The second is
that the existence of a cusp mechanism leaves
open
the question of how to
come up with a
general definition that
assigns a numerical ``order'' to each rigid framework.
While there have been several proposals \cite{tachi,gar,stachel,narwal},
each comes with difficulties
(see Section~\ref{sec:prvOrder} below)
and no consensus on the right definition has emerged.

From the point of view of engineering, there is some informal understanding that
when a framework has
higher order flexes, it should ``feel'' floppy, even if it is
mathematically rigid. One way to make sense of
how floppy it ``feels'' is to introduce an energy function, which is
a function that increases when bars
change length from their preferred lengths, and which is so-named
because of the connection to the deformation
energies of physical materials. For a first-order rigid framework,
such an energy function will have a
minimum at $\p$ with a Hessian that is positive definite,
hence, the energy increases quadratically in distance as the
framework moves away from $\p$. For frameworks
that are not first-order rigid, the energy is expected to increase
more slowly than quadratically. Precisely
how slowly has not yet been quantified,  and indeed one might expect
the behavior could depend on the
particular choice of energy function.

Our goal in this paper is to revisit the notion of higher-order rigidity and the connection to energy functions one can define on frameworks.
We aim to provide a link between the rigidity
properties of a framework, and a notion of the growth order of an
associated physically-based
energy function, and furthermore to explore the extent to which this
growth order depends on the
choice of energy function.
In particular
this allows us to propose a formal definition of the
rigidity order of a framework, which we
believe is  natural and  robust to previous concerns.
This approach also allows us to prove existing theorems in new ways, thus
providing further insight into these
theorems and in particular into their connection to properties of
physical interest.  In sum, we hope that this paper clarifies and
unifies an area of rigidity theory where
the literature to date has been somewhat complicated and confusing.

Our specific contributions are as follows.

Inspired by previous work~\cite{tachi,gar,stachel,narwal},
in Definition~\ref{defn:order},
we present a definition for the rigidity order of a framework. We
define the rigidity order of a framework $(G,\p)$ using
the growth order of an energy function $E$ near $\p$.  While the definition
might appear to depend on the energy $E$, we show that, in fact,
every $E$ in the class of ``stiff-bar'' energies, which includes all of the
energies commonly used to model fixed-length bars,
yields the same rigidity order for $(G,\p)$.

In Theorem~\ref{thm: order jk-flex},
we show that this rigidity order can be expressed in terms of
the existence and non-existence of certain $(j,k)$-flexes (to be
defined) of $(G,\p)$.  In general, this
characterization involves exploring
$(j,k)$ flexes over all positive integers $j$ and $k$.

Next,
we use our framework to
show that in some cases, we only need to explore
$(1,k)$ flexes (i.e., we can restrict to $j=1$). We also show that in some cases
checking for a $(1,k)$ flex
can be done efficiently.
In particular,
we provide  a novel
energy-based
proof of Theorem~\ref{thm:2or}, showing that if a framework has no
non-trivial
second-order flex, then it is rigid (and has rigidity order of $2$
under our energy based definition).
Moreover, when the dimension of the space of  non-trivial first-order
flex coefficients $\p'$ is only $1$, this
test can be done efficiently.
(The efficiency in this case essentially follows
  from ~\cite[Prop. 4.1.2]{pss};
see  Remark~\ref{rem:2comp}.)

The technical
difficulty we encounter is that, when $(G,\p)$ is
not first order rigid, $\p$ will be a degenerate critical point of the energy.
To deal with this  issue, we describe a novel 4th derivative test for
classifying
a degenerate critical point. This test generalizes ideas  contained in
Cushing~\cite{cushing} (Cushing's result treats only critical points where
the Hessian has nullity one), and it may be of independent interest
beyond rigidity theory.

In the rigidity literature, there is one special case in which the
naive generalization of second
order rigidity test can be shown to work.
V. Alexandrov  showed~\cite{alexImp} that,
when the dimension of the space of  non-trivial first-order flex coefficients $\p'$ is only $1$,
testing for third and higher order flexes \emph{can} be used to
establish rigidity.
In particular, he showed that, in this case, if there exists a $k$ so that there
is no non-trivial $k$th order flex, then
the framework is rigid (Theorem~\ref{thm:alex} below).
Alexandrov's test can be done
efficiently. This efficiency essentially follows from
~\cite{alexImp}, see Lemma~\ref{lem:kcomp}.

In this paper,
we provide a novel proof of Alexandrov's Theorem~\ref{thm:alex}
that is also based on our energy approach and  critical
point analysis.
We show that when it can be applied,
Alexandrov's test tells us the rigidity order
of the framework.
Our proof is informed by Cushing's ideas in \cite{cushing},
but different in the details.  We discuss the relationship below.

The study of structures that are rigid, but
only at ``high order'' is not only of mathematical
interest but also arises in applications.
While it is true that an engineer designing a practical and unyielding
bar-and-joint structure will almost certainly desire
first order rigidity, the design of rigid yet softer structures 
is also important with its study
dating back to 
Leonardo da Vinci~\cite{leonardo}
(see Figure~\ref{fig:k33} (right)).
At the same time, designing lattices and origami patterns with controlled soft modes leads to programmable metamaterials \cite{silverberg2014using,chen2018branches,zhai2021mechanical}. 
Moreover, such structures arise naturally
under various physical processes studied
in the sciences.
 In biological systems, 
 such structures naturally arise 
as cells stiffen under stress, migrate, or reshape~\cite{manning2024rigidity, hain2025optimizing}.
 In condensed matter physics, soft modes capture collective atomic motions in crystals and glasses, explaining phenomena such as crystal nucleation \cite{meng2010free,hk}, yielding under stress and the anomalous heat capacity of disordered solids \citep{mao2018maxwell}. 
 Understanding these modes provides powerful insights into statistical mechanics, failure and plasticity.

\subsection*{Acknowledgments}
We thank Dylan Thurston for  helpful conversations
and suggestions.
We thank Georg Nawratil for pointing out the
difference between trying to define rigidity orders
vs. flexibility orders.
This material is based upon work supported by the
National Science Foundation
(NSF) under Grant No. DMS-1929284 while the authors were in residence
at the Institute for
Computational and Experimental Research in Mathematics (ICERM) in
Providence, RI,
during the semester program ``Geometry of Materials, Packings and
Rigid Frameworks.''
M.H.C. acknowledges support from the Natural Sciences and Engineering
Research Council of Canada (NSERC), RGPIN-2023-04449 /
Cette recherche a été financée par le Conseil de recherches en
sciences naturelles et en génie du Canada (CRSNG). L.S.T. was 
partially supported by UKRI by the EPSRC Small Grants Scheme number 
UKRI1112.

\section{Rigidity - Summary and definitions}
We briefly recall the required rigidity definitions.

\begin{definition}
  We fix a dimension $d$.
  Let $n$ be a positive integer.
  A \defn{configuration} is an indexed set
  $\p:=\{\p_1,\ldots,\p_n\}$ of $n$ points, each in
  $\RR^d$. Let $G$ be a graph on $n$ vertices
  with $m:=|G|$ edges.
  A \defn{framework} is a pair $(G,\p)$.
  We assume there
  are no coincident points connected
  by an edge.

  Two frameworks, $(G,\p)$ and $(G,\q)$
  are \defn{equivalent} if for each edge $ij$ in $G$,
  we have $|\p_i-\p_j|=|\q_i-\q_j|$.
  Two configurations $\p$ and $\q$
  are \defn{congruent} if for each vertex pair
  $ij$,
  we have $|\p_i-\p_j|=|\q_i-\q_j|$.
  Congruent configurations arise due to isometries of
  $\RR^d$ acting on the points.

  We say that $(G,\p)$ is (locally)
  \defn{rigid} in $\RR^d$
  if there is a neighborhood $U$ of $\p$ so that
  for any $\q\in U$ where $(G,\q)$ is equivalent to
  $(G,\p)$, we have $\q$ congruent to $\p$.
  Otherwise we say that it is (locally)
  \defn{flexible}.
\end{definition}

Euclidean edge lengths are invariant with respect to congruence,
which means that every neighborhood of a configuration
$\p$ will contain other configurations $\q$ with the same
energy, preventing $\p$ from being a strict local minimum of an
edge-based energy.
To deal with this, we will remove congruences
by pinning down configurations (freezing degrees of freedom).
This is a
standard folklore operation in rigidity theory, but
some care is needed to deal with degenerate situations, where
the affine span of the framework is smaller than the space it is embedded in.

\begin{definition}\label{def: pinned configuration}
  Let $\p$ be a configuration of $n$ points in $\RR^d$
  and set $D$ to be the minimum of $n$ and $d$. 
  We say that $\p$ is \defn{pinned}
  if $\p_1$ is at the origin,
  and then, for $2\le i\le D$, $\p_i$ is in the
  span of the first $i-1$ coordinate vectors.

  Let  $\p$ be a configuration of $n$ points
  in $\RR^d$ with an
  $\ell$-dimensional affine
  span.
  We say $\p$ is \defn{lead-spanning} if
  $\p_1, \ldots, \p_{\ell+1}$ are affinely independent.
\end{definition}

Now we can define local rigidity for pinned frameworks.

\begin{definition}\label{def: pinned rigid}
  Fix a dimension $d$.  Let $\p$
  be a pinned and lead-spanning $d$-dimensional  configuration.
  A framework $(G,\p)$ is
  \defn{pinned rigid} in $\RR^d$
  if there is a neighborhood $U\ni \p$
  so that
  for any $\q\in U$ where $(G,\q)$ is equivalent to
  $(G,\p)$ and $\q$ is pinned,
  we have $\q = \p$.
  Otherwise we say that it is
  \defn{pinned flexible}.
\end{definition}

The next proposition, which is
folklore in the rigidity literature,
states that we can restrict our study
of rigidity to the pinned setting.
For completeness, we provide a
proof in  Appendix~\ref{sec:pin}.

\begin{proposition}\label{prop: pinned rigid}
  Fix a dimension $d$.
  Let $(G,\p)$ be a framework such that $\p$
  is lead-spanning.
  Let $\q$ be a pinned configuration that is
  congruent to $\p$.
  Then $(G,\p)$ is rigid in
  $\RR^d$ if and only if
  $(G,\q)$ is
  pinned rigid in $\RR^d$.
\end{proposition}
Proposition \ref{prop: pinned rigid} implies (by appropriate
relabeling) we can pin (using a similar procedure) \emph{any}
lead spanning subset of
$d+1$ vertices in $(G,\p)$ without
changing rigidity.
In this paper, we are interested
in the question of whether a given framework $(G,\p)$
is rigid.  We can relabel $\p$ (and $G$ in the same way)
and then apply a congruence to obtain a pinned framework
$(G',\q)$.  The above proposition says that $(G,\p)$ is
rigid if and only if $(G',\q)$ is pinned rigid, so rigidity and
pinned rigidity are equivalent questions.

Later on,
we will define other notions such as \emph{$(j,k)$-flexes} and
\emph{$(j,k)$-E-flexes} that are only defined for pinned frameworks.
In Appendix~\ref{sec:pin}, we show 
that whether or not these
exist is independent of the choice of
which vertices were pinned, as long as 
they are lead spanning.

Together these propositions imply that our main results,
which use energies to prove pinned rigidity, also prove
(unpinned) rigidity. Additionally the rigidity orders we define
are the same under any 
appropriate  choice of 
vertices pinned.
Hence:
\emph{in the remainder of the paper, when studying the rigidity
  of a framework $(G,\p)$,
  we will assume that $\p$ is
  lead-spanning and pinned. 
  All configurations,
  frameworks, trajectories, and flexes of different order that we consider are
are pinned (though not necessarily lead-spanning).}

\subsection{Rigidity and flexes}

To study higher order rigidity of a framework $(G,\p)$,
we will use families of analytic trajectories we call
$(j,k)$-flexes.  Informally, a $(j,k)$-flex starts moving at
order $j$ and preserves  the edge lengths through order $k$.
We will see that a $(1,1)$-flex coincides with the standard
rigidity-theoretic notion of an infinitesimal flex.  Now we
proceed to the formal definitions.
\begin{definition}\label{def: trajectory}
  Given a framework $(G,\p)$,
  a \defn{trajectory at $\p$}
  is a non-constant, mapping
  $\p(t)$ with $\p(0)=\p$,
  which is defined for $t \in [0,\eps]$, for some $\eps > 0$, and which 
  is a real analytic for all points in 
  $[0,\eps]$.
  When clear from the context, we will often drop the
  ``at $\p$'' phrase.

  We denote the $k$th derivative of $\p(t)$ at $t = 0$ by
  \ba
  \p^{(k)} :=  \left.\frac {d^k} {dt^k} \p(t)\right|_{t=0}.
  \ea

  We write 
\ba
\p(t) &=& \p + \frac{1}{j!} \p^{(j)}t^j + \cdots + 
\frac{1}{k!}\p^{(k)}t^k + \frac{1}{(k+1)!}\p^{(k+1)}t^{k+1} +
    \cdots
\ea
where $j$ is the first non-zero derivative of $\p(t)$.

We will also write 
\ba
\p(t) &=& \p + \a^{(j)}t^j + \cdots + 
\a^{(k)}t^k + \a^{(k+1)}t^{k+1} +
    \cdots
\ea
where $\a^{(k)}=\frac{1}{k!}\p^{(k)}$ is the $k$th Taylor
coefficient of $\p(t)$ at $t=0$.  
\end{definition}

In the above, we  use $\a^{(k)}$ to denote a coefficient, not a derivative. We use superscripts for $\a^{(k)}$
so as not to confuse it with $\a_k$, the $k$th point in the 
configuration $\a$.

\medskip

We now given the definition of a $(j,k)$ flex.
Define the following functions:
\ba
l_{ij}(\p) &:=& |\p_i-\p_j| \\
\l(\p) &:=& (l_{ij})_{ij\in G} \\
m_{ij}(\p) &:=& |\p_i-\p_j|^2 \\
\m(\p) &:=& (m_{ij})_{ij\in G}.
\ea
Here $\l(\p)$ and $\m(\p)$ are vectors in $\RR^{|G|}$
indexed by $\{ij\} \in G$.
So $\l(\p)$ (resp. $\m(\p)$) is
the vector of (resp. squared)
edge lengths of a framework $(G,\p)$.

\begin{definition}
  We say a real  valued $C^k$ function $f(t)$
  is \defn{$k$-vanishing} if its first $k$ derivatives vanish at $t=0$.
  A $C^k$ real vector valued function $\f(t)$
  is $k$-vanishing if all of its components
  are $k$-vanishing real valued functions.

  We say a real or vector valued $C^k$ function is
  \defn{$k$-active} if it is $k-1$-vanishing but not   $k$-vanishing.
\end{definition}

\begin{definition}
  Let $(G,\p)$ be a framework.
  Let $j,k$ be  integers $\ge 1$.
  A
  trajectory $\p(t)$ at $\p$ is a \defn{$(j,k)$-flex}
  of $(G,\p)$
  if
  $\p(t)$ is $j$-active
  (so $\p^{(j)}\neq0$)
  and
  $\m(\p(t))$ is $k$-vanishing.
\end{definition}

\begin{remark}
  For $i\in [1..j-1]$
  \[
    \frac{d^i}{dt^i}\m(\p(t))|_{t=0} = 0
  \]
  is true for any $j$-active trajectory.
  The additional
  restriction on a $(j,k)$-flex is that the derivatives
  of $\m(\p(t))$
  also
  vanish for $i\in [j,k]$.
\end{remark}

Because of this remark, only $(j,k)$ flexes with $k \geq j$ are
interesting, though the case $k<j$ sometimes occurs as a technical
step in proofs (e.g. the proof of Lemma \ref{lem42}).

\begin{remark}
  \label{rem:lorm}
  In this definition, we could instead use
  the vector of
  (unsquared) edge lengths $\l$,
  instead of $\m$, provided none of the lengths is 0.
  This follows from Lemma~\ref{lem:vanish1}, with $g=m_{ij}(\p(t))$
  for each $ij\in G$, and $f(g) = \sqrt{g}$. We choose to work with $\m$ in the definition because it leads to slightly simpler equations to solve for the flexes. 
\end{remark}

\begin{remark}
  \label{rem:irrFlex}
  This definition is similar to one of
  Stachel~\cite{stachel,narwal}
  (see also~\cite{stachel1999,sabitov,tarnai,gar}).
  Stachel also requires that a
  $(j,k)$ flex does not
  arise from a polynomial parameter substitution of a
  $(j',k')$ flex with $j'<j$.
  We will not do that here.
  For us, if $\p(t):=\p+\a^{(1)}t$ is a $(1,1)$ flex, then
  $\p+\a^{(1)}t^2+0t^3$ is a $(2,3)$ flex, and
  $\p+\a^{(1)}t^3+0t^4+0t^5$ is a $(3,5)$ flex.

\end{remark}

\begin{remark}
  \label{rem:trunc}
  Our definition of a $(j,k)$-flex is slightly different than some
  others that appear in the literature, because we allow
  $\p(t)$ to be an
  (analytic) $j$-active trajectory, and do not restrict it to being a
  polynomial of degree $k$.
  However, if
  \[
    \p(t) = \p + \a^{(j)}t^j + \cdots + \a^{(k)}t^k + \a^{(k+1)}t^{k+1} +
    \cdots
  \]
  is a $(j,k)$-flex, then so too is the truncation
  \[
    \p + \a^{(j)}t^j + \cdots + \a^{(k)}t^k,
  \]
  which is a polynomial.
\end{remark}

\begin{remark}

  Given a $(j,k)$-flex (with $k$ finite), the
  first $k$ derivatives 
  form
  an ordered sequence
  $(\p, \p',\p'',\p''',\ldots,\p^{(k)})$ of  configurations
  such that:
  \bna
  \p^{(i)} = 0 & \text{for all $1\le i\le j-1$}; \\
  \label{eq:flex}
  \sum_{a=0}^{l} \binom{l}{a}
  ({\p}^{(a)}_v-{\p}^{(a)}_w)\cdot({\p}^{(l-a)}_v-{\p}^{(l-a)}_w) =0
  & \text{for all $1\le l\le k$}.
  \ena
  This can be seen by taking the first $k$ derivatives of
  $\m(\p(t))$ (see, e.g., \cite{connelly-second}).

  For example, for a $(1,4)$-flex, the
  coefficients of the
  first $4$ terms of its Taylor expansion satisfy, for all edges $vw$,
  the system of nonlinear equations (a factor of 2 has been removed
  from all equations):
  \begin{align*}
    (\p_v-\p_w)\cdot(\p'_v-\p'_w) &=0\\
    (\p_v-\p_w)\cdot(\p''_v-\p''_w) + (\p_v'-\p_w')\cdot(\p'_v-\p'_w) &=0\\
    (\p_v-\p_w)\cdot(\p'''_v-\p'''_w) + 3(\p_v'-\p_w')\cdot(\p''_v-\p''_w) &=0\\
    (\p_v-\p_w)\cdot(\p''''_v-\p''''_w)
    + 4(\p_v'-\p_w')\cdot(\p'''_v-\p'''_w)
    + 3(\p_v''-\p_w'')\cdot(\p''_v-\p''_w)
    &= 0.
  \end{align*}
In some rigidity literature~\cite{connelly-second}, a flex is defined to be the sequence 
  $(\p, \p',\p'',\p''',\ldots,\p^{(k)})$  itself.
\end{remark}

If a framework $(G,\p)$ is flexible then it must have an
analytic finite flex $\p(t)$ \cite{Gluck,milnor}. We
could define such a $\p(t)$ to be a $(j,\infty)$-flex,
for the appropriately chosen $j$.  We cannot, however,
reparameterize the $(j,\infty)$-flex $\p(t)$ to be a
$(1,\infty)$-flex $\q(s) := \p(s^j)$ because
$\q(s)$ has a Puiseux series expansion in $s$ (as
opposed to a power series), and it may not even be $C^2$
at $s=0$.  Hence, to develop sufficient conditions for
rigidity in terms of $(j,k)$-flexes, we will need to
consider all possibilities for $j$ and $k$.

\subsection{Sufficient conditions for rigidity}
Now we review some sufficient conditions for rigidity that appear in the
literature and rephrase them in terms of $(j,k)$ flexes.

\subsubsection{First-order rigidity}
In the rigidity literature, a lead-spanning pinned framework
that does not have a $(1,1)$-flex is called \emph{first order rigid} or
\emph{infinitesimally rigid}.  Checking for the existence of a $(1,1)$-flex can be done efficiently using linear algebra. Later, we define a the rigidity order
of a framework.
A first-order rigid framework will have, in our definition,
rigidity order one.
\begin{theorem}
  \label{thm:1or}
  If a framework $(G,\p)$ is first-order rigid,
  then it is rigid.
\end{theorem}

There are many, many,  ways to prove
Theorem~\ref{thm:1or} (see \cite{Asimow-Roth-I,connelly-second,Gluck}
for some examples).
In~\cite{connelly-second}, it is proven
as follows: assume $(G,\p)$ has a finite flex $\p(t)=\p+\a^{(j)}t^j + \cdots$
for some $j$ and coefficient
$\a^{(j)}$.
This gives rise to the $(j,j)$-flex $\p+\a^{(j)}t^j$,
from which it follows that 
$\q(s):=\p+\a^{(j)}s$ is a
$(1,1)$ flex, with $s=t^j$.

For later, we need some specific facts and notation about
$(1,1)$ flexes.

\begin{definition}
  The space of first derivatives $\p'$ of the $(1,1)$  flexes
  for (pinned) $(G,\p)$, together with the zero vector,
  is a linear
  subspace of (pinned) configuration space,
  which we will denote as $K$ (for kernel).
  We will also chose a complementary space in (pinned)  configuration space
  and
  denote it $\overline{K}$.
\end{definition}

Note that $K$ is the kernel of the
so-called (pinned) rigidity matrix of
$(G,\p)$. The rigidity matrix $R(\p)$ is such that the
linear system of equations
$$(\p_v - \p_w)\cdot (\p'_v - \p'_w)=0 \qquad
\text{for all edges } vw,$$
can be written as $R(\p)\p'=0$. Pinning certain degrees of freedom is
equivalent to removing the corresponding columns of $R(\p)$ and the
corresponding elements of $\p'$.


\subsubsection{Second-order rigidity}
The next rigidity test we consider
is second
order rigidity.  A lead-spanning and pinned framework $(G,\p)$ is called
second order rigid if it does not have a $(1,2)$-flex.
A second order rigid framework will have rigidity order
two, as we define it below.  The main fact about second order
rigidity is that it implies rigidity.
\begin{theorem}[{\cite{connelly-second}}]
  \label{thm:2or}
  If a framework $(G,\p)$ is second order rigid, then it is
  rigid.
\end{theorem}

In this paper, we will give a new proof of Theorem~\ref{thm:2or}
by analyzing critical points of
a physically-motivated energy function. It will turn out that
the critical point in question will be degenerate,
so we develop a  novel general
4th derivative test that can certify a strict local
minimum.  Certifying rigidity using the new
$4$th derivative test is equivalent to ruling out
a $(1,2)$-flex, making the connection to second-order
rigidity.
We will actually prove a more precise statement than
Theorem \ref{thm:2or}, namely Theorem~\ref{thm:main} below.

\begin{remark}
  \label{rem:2comp}
  If 
  $\p+ \a^{(1)}t + \a^{(2)}t^2$
  is a $(1,2)$-flex and $\p+\b^{(1)} t$ is
  a $(1,1)$-flex, then
  $\p+ \a^{(1)}t + (\a^{(2)}+\b^{(1)})t^2$
  will also be a $(1,2)$-flex~\cite[Prop. 4.1.2]{pss}
  Hence, a framework has a $(1,2)$-flex if
  and only if it has one with $\a^{(2)}\in \overline{K}$.
\end{remark}

\subsubsection{Higher order rigidity}
The rigidity tests described above for first order rigidity
second order rigidity are the classical ones. Moreover, they nest nicely:
first order rigidity
implies second order rigidity.
It has, however, long been known
that there are frameworks that are rigid but not second order
rigid.  Hence, there has been interest in defining a notion
of ``higher order rigidity'' using flexes with the property that
every rigid framework is ``rigid at some order'' and that rigidity at
``lower order''
implies rigidity at ``higher order''.

In \cite{connelly-second} the following
natural proposal is discussed: a framework $(G,\p)$ is ``$k$th-order rigid'' if
it does not have $(1,k)$-flex.  It is left as a question in
\cite{connelly-second}
whether a framework that is $k$th-order rigid, in this sense, must be rigid.
In \cite{conSer}, a negative answer is given, in the form of a framework that
is ``$3$rd-order rigid'', but not rigid.
The example in \cite{conSer} is the
so-called ``double-Watt mechanism'', which has a cusp in its configuration
space.  The precise statement (translated to our terminology) is as follows.
\begin{theorem}[{\cite{conSer}}]\label{thm:3or}
  There exists a framework $(G,\p)$ that
  has a $(2,\infty)$-flex but no $(1,3)$-flex.
\end{theorem}
Below, we define the rigidity order of a framework $(G,\p)$.  Our
definition has property that a framework $(G,\p)$ is rigid if
and only if it has a finite rigidity order.  Necessarily,
the rigidity order of  $(G,\p)$ being equal to $k$ is not
characterized by $(G,\p)$ not having a $(1,k)$-flex.  The
two concepts are different.

However, there
is a special case in which our rigidity order coincides with
the notion of $k$th-order rigidity as described in \cite{connelly-second}.
The following result of V. Alexandrov \cite{alexImp} states that,
if $\dim(K) = 1$, ruling out a $(1,k)$-flex {\em does} imply rigidity.
\begin{theorem}[{\cite[Theorem 10]{alexImp}}]
  \label{thm:alex}
  Let $(G,\p)$ be a framework such that
  $\dim(K)=1$.
  If there exists a $k$ such that the $(G,\p)$ has no $(1,k)$-flex,
  then $(G,\p)$ is rigid.
\end{theorem}
(There is no contradiction with Theorem~\ref{thm:3or}
  as the double-Watt mechanism described in \cite{conSer} has
$\dim(K) = 2$.)

In this paper, we will give a new proof of  Theorem~\ref{thm:alex}.  Like
our new proof of Theorem \ref{thm:2or}, the approach is to
analyze critical points of an energy.
We will actually prove the more precise statement than Theorem \ref{thm:alex},
namely Theorem~\ref{thm:main2} below.
For what follows, we will need this next
lemma, which is similar to Remark~\ref{rem:2comp}.
This result is also in
~\cite[Section 3.4]{gar} and
~\cite{alexImp}.

\begin{lemma}
  \label{lem:kcomp}
  Given a 1-active trajectory $\q(s)$ with $\q(0)=\p$, let $Z = \mbox{span}(\q')$, and let $\overline{Z}$ be a space complementary to $Z$. Then, for each integer $k$, there is an analytic reparameterization $\p(t) = \q(s(t))$ with $s'(0)\neq 0$ such that $\p'\in Z$ and
    $\p\me{j}\in \overline{Z}$ for $2\leq j\leq k$. Furthermore, for fixed $s'(0)$, the reparameterization is unique. 
\end{lemma}

The following corollary immediately follows. 
\begin{corollary}
    If $\mbox{dim}(K)=1$, then any $(1,k)$ flex 
    $\q(s)=\p+\q's+\frac{1}{2!}\q''s^2 + \cdots$ 
    can be analytically reparameterized into the form
    $\p(t)=\p+\p't+\frac{1}{2!}\p''t^2 + \cdots$ 
    with $\p\me{j}\in \bar{K}$ for $2\leq j\leq k$.
    Furthermore, for fixed $\p'$, the reparameterization is unique. 
\end{corollary}

The proof of Lemma \ref{lem:kcomp} proceeds by solving for the derivatives of $s(t)$ at $t=0$ using the chain rule.  Notice that if we look for such a reparameterization $s(t)$, then the chain rule implies 
$\p' = \q's'.$
Hence, if we choose $s'=1$, then $\p'=\q'$. 
After two iterations of the chain rule we have
\[
\p'' = \q''(s')^2 + \q's''. 
\]
We know that $\q'' = \lambda \q' + \w_1$ for some $\w_1 \in \bar Z$. Hence, if we choose $s'' = -\lambda$, then we'll have $\p''\in \bar Z$. Furthermore, this is the unique such choice of $s''$. After three iterations of the chain rule we have
\[
\p''' = \q'''(s')^3 + 3\q''s's'' + \q's'''.
\]
We could proceed by choosing $s'''$ to compensate the component of $\q'''+ 3\q''s''$ that lies in $Z$. However, the presentation is easier by assuming we already have a trajectory with $\q''\in \bar Z$ (obtained for example by the reparameterization at the second step), and then we only have to choose $s'''$ to compensate the $\q'''$ term. In this second version, we construct $s(t)$ by iterating the parameterizations obtained at each step. 

\begin{proof}\emph{(Proof of Lemma \ref{lem:kcomp})}
    We show the following statement: suppose $\q(s)$ satisfies $\q\me{j}\in \bar Z$ for $2\leq j\leq k-1$. Then there is an analytic reparameterization of the form $s(t) = t + \frac{1}{k!}s\me{k}t^k$ such that $\p\me{j}\in \bar Z$ for $2\leq j\leq k$. The statement of the Lemma then follows by iteratively applying reparameterizations, up to the desired value of $k$. The proof uses a general formula for multiple iterations of the chain rule, called Fa\`{a} di Bruno's formula, which is presented and discussed in more detail later (Eqn. \eqref{eq:faa}).

    To show the statement, we look for a parameterization of the form $s(t) = t + \frac{1}{k!}s\me{k}t^k$ for some coefficient $s\me{k}$. First note that 
    because $s'=1$ and $s\me{j}=0$ for $2\leq j\leq k-1$, and using
    \eqref{eq:faa}, we have $\p\me{j}=\q\me{j}$ for $2\leq j\leq k-1$. Consider $\p\me{k}$. By \eqref{eq:faa} we also have,
    \[
   \p\me{k} = \q\me{k}(s')^k + \q's\me{k}
    = \q\me{k} + \q's\me{k}. 
    \]
    We know that $\q\me{k} = \lambda_k\q' + \w_k$ for some unique number $\lambda_k\in \R$ and some unique vector $\w_k\in \bar Z$. 
    Therefore $\q\me{k} = (\lambda_k-s\me{k})\q' + \w_k$.
    Choose $s\me{k} = -\lambda_k$ to get $\p\me{k} = \w_k\in \bar Z$. By the uniqueness of the decomposition of $\q\me{k}$, the choice of $s\me{k}$ is unique. 

    The uniqueness of the overall parameterization follows because every parameterization can be constructed in this iterative way, and we have shown that each step of the iteration is uniquely defined. 
\end{proof}

A consequence of Theorem \ref{thm:alex} (and Lemma~\ref{lem:kcomp})
is that, when $\dim(K)=1$,
certifying rigidity by ruling out a $(1,k)$-flex for a fixed $k$
can be done efficiently.
This is because, when $\dim(K) = 1$, for any fixed $k$, we can determine
whether there is a $(1,k)$-flex by iteratively solving a sequence of $k$ linear systems.

To see why, let us first
consider the
general problem of finding a $(1,2)$-flex.  The requires solving
for $\p'$ and $\p''$,
in the equations
\[
  (\p_v - \p_w)\cdot (\p'_v - \p'_w) = 0 \qquad
  (\p_v - \p_w)\cdot (\p''_v - \p''_w)
  =
  -(\p'_v - \p'_w)\cdot (\p'_v - \p'_w),
\]
for all edges $vw$ of $G$.  This is a nonlinear problem, and there is no known
efficient algorithm.  However, in the special case that $\dim(K) =
1$, the solutions
for $\p'$ are scalar multiples of each other.
Meanwhile,  $\p(t) = \p + \p't + \frac{1}{2}\p''t^2$
is a $(1,2)$-flex of $(G,\p)$ if and only if, for any non-zero scalar $\alpha$,
$\q(t) = \p + \alpha\p't + \alpha^2\frac{1}{2}\p''t^2$ is also a $(1,2)$-flex.
Thus we can fix
a scale (and signs) for $\p'$ at will.

This means we can
easily test for the existence of a $(1,2)$-flex: first solve a linear system
to find any $(1,1)$ flex $\p'$; then fix $\p'$ and try to solve the
linear  system
\[
  (\p_v - \p_w)\cdot (\p''_v - \p''_w) =     -(\p'_v - \p'_w)\cdot
  (\p'_v - \p'_w)
  \qquad \text{for all edges $vw$ of $G$}.
\]
This is linear as the right hand side has been fixed.
As discussed above, there is a solution iff there is a $(1,2)$-flex,
since they are
all related by scaling.
Meanwhile note that $K$ is the kernel
of the coefficient matrix of the second linear
system, which is the same as the one used to
define $(1,1)$ flexes\footnote{Readers familiar with rigidity theory
will recognize it as the rigidity matrix of $(G,\p)$.}.
Thus when there is a solution, it has a unique solution
in $\overline{K}$.

The algorithm for $k\ge 3$ is to continue the pattern.  Suppose that we have
found a $(1,k)$-flex, $\p(t) = \p + \p' t + \cdots + \frac{1}{k!}\p^{(k)}$ with the
property that $\p^{(i)}\in \overline{K}$ for $2\le i\le k$.  With these
choices fixed, the system \eqref{eq:flex} for $k+1$ is linear system
in $\p^{(k+1)}$,
namely
\[
  (\p_v - \p_w)\cdot (\p^{(k+1)}_v - \p^{(k+1)}_w)
  =-\frac{1}{2}\sum_{a=1}^{l} \binom{l}{a}
  ({\p}^{(a)}_v-{\p}^{(a)}_w)\cdot({\p}^{(l-a)}_v-{\p}^{(l-a)}_w) =0
  \qquad \text{for all $1\le l\le k$},
\]
(where the kernel of the coefficient matrix is still $K$).
If there is a solution, we get
a unique solution in $\overline{K}$ (see also Lemma~\ref{lem:kcomp}).
If there is no solution in $\overline{K}$, then
from Lemma~\ref{lem:kcomp}
we get a certificate that there is no
$(1,k+1)$-flex.

By induction, using Lemma \ref{lem:kcomp} and the preceding argument,
once we fix $\p'$, if there is a $(1,k)$-flex, we can find it 
uniquely by searching
for $\p^{(i)}$ in $\overline{K}$ one step at a time,  without having to back up.
Combining the algorithm with Theorem \ref{thm:alex}, we get:
\begin{theorem}
  \label{thm:eff}
  Let $(G,\p)$ be a framework with
  $\dim(K)=1$.  For any fixed $k\in \mathbb{N}$, there is an
  efficient algorithm that either certifies that $(G,\p)$ is
  rigid or that it has a $(1,k)$-flex.
\end{theorem}

We don't know in advance which $k$ to pick, so we don't get an
algorithm that efficiently decides rigidity, even in this special
case.  However,
the results of this paper indicate that, if, for large $k$,
$(G,\p)$ has a $(1,k)$-flex, it will feel very ``soft''.

\section{Rigidity and energy}
\label{sec:enOr}

Energy formulations have a long history of
being used to study the
rigidity of structures (see e.g. ~\cite{koiter,salerno,gar}).

\begin{definition}
  \label{def:energy}
  Let $(G,\p)$ be a lead-spanning and pinned framework.
  For each edge $ij \in G$,
  let $E_{ij}(l) : \RR\to \RR$ be a univariate
  function. We define the \defn{edge-based
  energy}
  \ba
  E(\q) := \sum_{ij \in G} E_{ij}(l_{ij}(\q))
  \ea
  where $\q$ is a pinned configuration.

  We say that an edge-based
  energy $E(\q)$ is \defn{bar-like}
  at $\p$,
  if: for each edge $ij\in E(G)$,
  $$d_{ij} := |\p_i - \p_j| \neq 0;$$
  each
  $E_{ij}$ is analytic at $d_{ij}$; and each $E_{ij}$
  has a strict local minimum at $d_{ij}$.  These conditions
  imply that
  \ba
  \frac{d}{dl}E_{ij}|_{l=d_{ij}}&=&0.
  \ea

  We say that an edge-based energy $E(\q)$
  is a \defn{stiff-bar energy} at $\p$,
  if $E$ is bar-like at $\p$ and, in addition, each
  $E_{ij}$ has positive curvature at $d_{ij}$;
  i.e.,
  \ba
  \frac{d^2}{dl^2}E_{ij}|_{l=d_{ij}}&>&0.
  \ea
  When clear from the context, we will often drop the
  ``at $\p$'' phrase.
\end{definition}

\begin{remark}
  Because $l_{ij}(\q)$ involves a square root,
  we need to assume that $(G,\p)$ has no zero length edges, in order
  for a stiff-bar energy $E(\q)$ to be analytic at $\p$.
\end{remark}

Here are some examples of stiff bar energies:
\begin{enumerate}
  \item \emph{Harmonic spring potential.} This energy treats each
    edge as a harmonic spring, with spring constant $k_{ij} > 0$ and
    rest length $d_{ij} > 0$:
    \[ E_{ij}(l) = \frac{1}{2}k_{ij}(l-d_{ij})^2.
    \]
  \item \emph{Algebraic energy.} This energy avoids the square root
    needed in computing $l_{ij}(\q)$:
    \[ E_{ij}(l) = \frac{1}{2}k_{ij}(l^2-d_{ij}^2)^2.
    \]
  \item \emph{Lennard-Jones potential.} This energy is commonly used
    in molecular dynamics simulations to model interactions between
    atoms \cite{LJ,LJnew}. The parameters $\sigma_{ij},\epsilon_{ij}$
    govern the location of the minimum, and the depth of the
    potential, respectively:
    \[
      E_{ij}(l) = 4\epsilon_{ij} \left(
        \left(\frac{\sigma_{ij}}{l}\right)^{12} -
      \left(\frac{\sigma_{ij}}{l}\right)^6\right).
    \]
  \item \emph{Morse potential.} This potential is used to model
    interactions between bodies with very short-ranged attractive
    interactions \cite{morse,morse2}. The parameter $d_{ij}$ is the
    location of the minimum, $D_{ij}$ is the depth, and $a_{ij}$
    governs the range of the attractive part of the energy:
    \[ E_{ij}(l) = D_{ij}(1-e^{-a_{ij}(l-d_{ij})})^2.\]
\end{enumerate}

\begin{remark}
In this paper, we  define our stiff bar
energies using terms of the form
$E_{ij}(l_{ij}(\q))$, i.e., functions of lengths.
We do this since many energies in various applications domains are
typically written this way.  We could just
as easily replace these terms  with functions
of squared lengths instead.
In particular, if we let 
$F_{ij}(m) := E_{ij}(\sqrt{m})$ then 
$E_{ij}(l_{ij}(\q))=F_{ij}(m_{ij}(\q))$.
And if $E_{ij}$ has a strict local minimum
with positive curvature at $d_{ij}$
(which is non-zero by assumption),
then $F_{ij}$ will have a strict local minimum
with positive curvature at $d^2_{ij}$ (see e.g., Remark~\ref{rem:writtenOut}).
\end{remark}

Let us record a simple connection between bar-like energies and (pinned) rigidity.  We
give a proof for completeness.
\begin{theorem}\label{thm: E min and rigidity}
  Let $E$ be an energy that is bar-like at $\p$.  Then
  $\p$ is a local minimum of $E$.  Moreover,
  a framework $(G,\p)$ is rigid if and only if $\p$
  is a strict local minimum of $E$.
\end{theorem}
\begin{proof}
  Fix an edge $ij\in E(G)$.  Because $E$ is bar-like, there is a
  neighborhood $U\in \p$ such that,
  for all $\q \in U$, $E_{ij}(\q)\ge E_{ij}(\p)$.  Comparing $E(\p)$
  and $E(\q)$ term-by-term
  shows that, for all $\q\in U$, $E(\q) \ge E(\p)$.  This shows that
  $\p$ is a local minimum of
  $E$.

  For the second part, we observe that, for the same neighborhood
  $U$, $E(\q) = E(\p)$ if and only
  if $E_{ij}(\q) = E_{ij}(\p)$ for all edges $ij\in E(G)$.  Because
  $E_{ij}$ is a function of
  edge length that has a local minium at the edge lengths of
  $(G,\p)$, this implies that
  $E(\q) = E(\p)$ if and only if $(G,\q)$ has the same edge lengths
  as $(G,\p)$. The second
  part now follows: $\p$ is a strict local minimum if and only if,
  for every neighborhood $U\supseteq V\ni \p$,
  $E(\q) = E(\p)$ implies that $\q = \p$; by the previous discussion,
  this is equivalent to
  $(G,\q)$ and $(G,\p)$ having the same edges implying that $\p = \q$.
\end{proof}
In this statement, it is important that we are considering
pinned configurations.  Otherwise, $\p$ will never be a
strict local minimum of $E$, because $E$ is
invariant to congruences.

\begin{remark}
  When the energy $E$ is not bar-like, i.e.,
  $\frac{d}{dl}E_{ij}|_{l=d_{ij}} \neq0$
  for some edges $ij$, $\p$ is not necessarily a local minimum of $E$
  when $(G,\p)$ is rigid.  However, some energies that
  are not bar-like can also be used to prove rigidity.
  Using energies that are not bar-like to prove rigidity is
  related to the notion of prestress stability \cite{pss},
  and we extend our results incorporate non-bar-like energies
  in an upcoming companion paper~\cite{HOpre}.
\end{remark}

\subsection{Energy growth}
When a framework is rigid,  it can
be informative to study at what order an energy
changes as we move away from $\p$.
These next definitions formalize the
growth rate of $E$ as a function of distance from $\p$.
\begin{definition}
  Let $E$ be a function
  over (pinned) configuration space
  with
  a local minimum at $\p$.
  Let $s > 0$ be a real number.

  We say that $E$ grows \defn{always-$s$-quickly} at $\p$ if
  there is some $c_1>0$
  and a  neighborhood $U$ of $\p$
  so that for all
  $\q$
  in $U$, we have
  \ba
  E(\q) -E(\p) &\geq& c_1|\q-\p|^s .
  \ea

  We say that $E$ grows \defn{sometimes-$s$-slowly} at $\p$  if
  there exists a trajectory $\p(t)$ at $\p$,
  and constants $c_2,\delta>0$,
  such that for $t \in [0,\delta]$
  \ba
  E(\p(t)) -E(\p)&\leq& c_2|\p(t)-\p|^s .
  \ea
  We say that $E$ grows \defn{sometimes-$\infty$-slowly} at $\p$  if
  there exists a
  trajectory at $\p$,
  $\p(t)$ so that
  \ba
  E(\p(t)) -E(\p) =0 .
  \ea
  We say that $E$ grows \defn{$s$-tightly} at $\p$  if
  it grows sometimes-$s$-slowly and always-$s$-quickly at $\p$.
  In this case we say
  that $s$ is the \defn{tight growth order} for $E$ at $\p$.
  When clear from the context, we will often drop the
  ``at $\p$'' phrase.
\end{definition}
Since the definitions of growth involve trajectories, it is not
a priori clear whether a tight growth order exists or
what $s$ can be.  An important fact is that, for analytic
functions, the situation is very nice.

\begin{theorem}[\cite{netto}]\label{thm: netto}
  Let $E$ be a function $E : \mathbb{R}^N\to \mathbb{R}$ which is
  analytic at a local minimum
  at $\p$.  Then
  \begin{enumerate}[(1)]
    \item For every trajectory $\p(t)$ at $\p$, there is a  \defn{growth order}
      $s_\p\in \{\mathbb{Q},\infty\}$ of $E(\p(t))$, i.e., there are
      constants $c_1, c_2 > 0$ such that
      \[
        c_1|\p(t) - \p|^{s_\p} \le
        E(\p(t)) - E(\p)
        \le c_2 |\p(t) - \p|^{s_\p};
      \]
    \item If $\p$ is a strict local minimum, $E$ grows $s$-tightly
      for a rational number $s$,
      and there is a trajectory $\p(t)$ at $\p$ so that $E(\p(t))$
      has growth order $s$;
    \item If $\p$ is a weak local minimum, there is a trajectory
      $\p(t)$ at $\p$ along which
      $E(\p(t))$ is constant (i.e., has $s_\p = \infty$).
  \end{enumerate}
\end{theorem}

In (2),
let $m(r)$, be the function that
expresses the minimal value of
$E(\q)-E(\p)$
as a function the radius $r:=|\q-\p|$.
Then
the quantity $s$ is simply the leading exponent in
the Puiseux series for $m(r)$ at $r=0$.
Some remarks are in order:
\begin{remark} \label{remark:netto}
  \begin{enumerate}[(1)]
    \item Every flexible framework has an analytic finite flex and
      hence $E$ grows sometimes-$\infty$-slowly at $\p$. So, the tight growth
      order of any bar-like energy $E$ at a framework $(G,\p)$ is finite if and
      only if $(G,\p)$ is rigid.
    \item Even if $E$ is a degree $d$ polynomial, $s$ can be much
      larger than $d$, and rigidity
      theory provides examples.
      In particular, the ``algebraic energy'' described in
      Definition~\ref{def:energy}
      is a degree $4$ polynomial,
      but, as we will show, at any framework with a $(1,k)$-flex,
      this energy grows sometimes-$2k$-slowly.
      There are such frameworks for unbounded $k$ \cite{leonardo}.
    \item The tight growth order $s$ of an analytic function $E$ at a
      strict local minimum $\p$
      does not need to be an integer \cite{netto}.
  \end{enumerate}
\end{remark}

We now introduce a notion that connects rigidity and
energy functions.
\begin{definition}\label{def: f-flex}
  Let $\f(\q)$
  be an  vector valued analytic function on (pinned) configuration
  space in some neighborhood of interest.
  Suppose $\p(t)$ is a $j$-active trajectory at $\p$ and
  $\f(\p(t))$ is $k$-vanishing.
  Then we say that $\p(t)$ is a
  \defn{$(j,k)$ $\f$-flex}.
\end{definition}

\begin{remark}
  Definition \ref{def: f-flex} generalizes the notion of a $(j,k)$-flex of a
  framework $(G,\p)$, since a $(j,k)$-flex may also be called a
  $(j,k)$ $\m$-flex.
\end{remark}
In this paper, we will be mainly be interested in
E-flexes, where $E(\q)$ is some stiff bar energy at $\p$. 

The following lemma 
gives us the connection between the existence of a
$(j,2k)$ E-flex and sometimes-slow growth.

\begin{lemma}
  \label{lem:eflex}
  Let $E$ be a function which is analytic at a local minimum $\p$, and suppose there exists
  a $(j,2k)$ E-flex
  $\p(t)$ at $\p$.
  Then E grows sometimes-s-slowly where $s=\frac{2k+2}{j}$.
\end{lemma}
\begin{proof}

  Since $\p$ is a  local minimum of $E$,
  a $(j,2k)$ E-flex must also be $(j,2k+1)$ E-flex.
  This implies that
  $E(\p(t))$
  is $(2k+1)$-vanishing, as an odd leading order would imply that the
  function decreases locally in some direction.

  So if we define
  \[
    a_{2k+2} =
    \frac{1}{(2k+2)!}\left.\frac{d^{2k+2}}
    {dt^{2k+2}}E(\p(t))\right|_{t=0}
  \]
  the univariate Taylor's theorem with Peano remainder gives
  \[
    E(\p(t)) - E(\p) = a_{2k+2}t^{2k + 2} + o(t^{2k + 2}).
  \]
  Because $E(\p(t))$ is non-decreasing,
  $a_{2k+2}$ is non-negative.  If
  $a_{2k+2} = 0$, then, for small enough $t$,
  \[
    E(\p(t)) - E(\p) \le t^{2k+2};
  \]
  otherwise, for small enough $t$, the $o(t^{2k+2})$ term is bounded
  in magnitude by $a_{2k+2}t^{2k+2}$, and so
  \[
    E(\p(t)) - E(\p) \le 2a_{2k+2}t^{2k+2}.
  \]
  Combining the cases, we set $A = \max\{1,2a_{2k+2}\}$
  \bna
  \label{eq:top}
  E(\p(t)) - E(\p) \le At^{2k+2}.
  \ena
  This upper-bounds the growth of the energy in terms of $t$.
  Now we get a bound on the growth of the energy in terms of $|\p(t)|$.
  From the definition of a $(j,2k)$
  E-flex, the trajectory
  $\p(t)$ must be $j$-active.  This, then, implies that
  $|\p(t) - \p|^2$ is is $2j$-active, because this squared length is
  a homogeneous quadratic polynomial  in $t$.  If we set
  \[
    b_{2j} =
    \frac{1}{(2j)!}
    \left.
    \frac{d^{2j}}{dt^{2j}} |\p(t) - \p|^2
    \right|_{t=0},
  \]
  then $2j$-activity implies that $b_{2j}\neq 0$, and
  then, because $|\p(t) - \p|^2$ is increasing at $t=0$,
  $b_{2j} > 0$.
  Taylor's Theorem implies that
  \[
    |\p(t) - \p|^2 = b_{2j}t^{2j} + o(t^{2j}).
  \]
  For sufficiently small $t$, the magnitude of the
  $o(t^{2j})$ is bounded by $\frac{1}{2}b_{2j}$, and so,
  for $t$ in the same range,
  \bna
  \label{eq:bot}
  |\p(t) - \p|^2 \ge \frac{b_{2j}}{2} t^{2j}.
  \ena
  Combining Equations (\ref{eq:top}) and (\ref{eq:bot}),
  for small enough $t$,
  \[
    E(\p(t)) - E(\p) \le A
    \left(\frac{2}{b_{2j}}\right)^{\frac{2k+2}{j}}|\p(t) - \p|^{(2k+2)/j}.
  \]
\end{proof}
\begin{remark}\label{rem:eflex}
  The proof of Lemma \ref{lem:eflex} shows more, which we note for
  later.  If the constant $a_{2k+2}$ appearing in the proof is
  non-zero, then it must be positive, and so the $(j,2k)$-E-flex
  $\p(t)$ has a  growth order of $\frac{2k + 2}{j}$.
\end{remark}

Let the radius be defined as $r(t):=|\p(t)-\p|$.
This has an inverse $t(r)$ in a neighborhood of $0$.
Consider the
function $\delta E(r) :=E(\p(t(r)))-E(\p)$.
The quantity
$\frac{2k+2}{j}$ in Lemma~\ref{lem:eflex}
can be though of as a lower bound on the leading exponent of
the Puiseux series for $\delta E(r)$ at $r=0$.

The upshot of Lemma \ref{lem:eflex} is that the existence of
a $(j,2k)$ $E$-flex gives a lower bound on $s$, the order
of growth of $E$ at $\p$. It turns out that the maximum over all these lower bounds, gives the growth order of $E$.

\begin{theorem}\label{thm:se}
Let $E$ be a function which is analytic at a strict local minimum $\p$. Then the tight order of growth of $E$ is given by 
\begin{equation}\label{sEflex}
    s_E =  \max \left\{ \frac{2k+2}{j} : \text{$(G,\p)$
    has a $(j,2k)$ E-flex}\right\}.
  \end{equation}
\end{theorem}

\begin{proof}
    $E$ has tight growth order of $s_E\in \mathbb Q$, as
  established in Theorem \ref{thm: netto} and Remark
  \ref{remark:netto}.  By hypothesis, $E$ grows
  always $s_E$-quickly at $\p$.

  Suppose
  now that $\p(t)$ is a  $(j,2k)$ E-flex.  Lemma
  \ref{lem:eflex} then implies that $E$ grows sometimes
  $\frac{2k+2}{j}$ slowly.
  Thus we get the inequality
  $\frac{2k+2}{j} \le s_E$.
  
  Next we show that there must be a
  $(j,2k)$ E-flex such that
  $\frac{2k+2}{j}=s_E$.
  From the definition of tight growth
  there
  is a trajectory
  $\q(t) = \p + \sum_{i=j_0}^\infty \a^{(i)}t^i$,
  with $\a^{(j_0)}\neq 0$,
  so that $E(\q(t))$ has  a growth order of $s_E$.
  Because $(G,\p)$ is a strict local minimum,
  $E$ is increasing along $\q(t)$, which implies that the
  Taylor expansion of $E(\q(t))$ at zero is of the form
  \[
    E(\q(t)) = E(\p)+\sum_{i=2k_0 + 2}^\infty a_i t^i;
  \]
  with $a_{2k_0+2}>0$.
  We have $\q(t)$ is a $(j_0,2k_0)$ E-flex of $(G,\p)$.
  From Lemma \ref{lem:eflex} and Remark \ref{rem:eflex},
  $\q(t)$ witnesses that $E$ grows
  sometimes at order $\frac{2k_0 + 2}{j_0}$.
  Hence  $s_E = \frac{2k_0 + 2}{j_0}$.
  (This establishes that we obtain a
  maximum in \eqref{sEflex}, not just a supremum.)
\end{proof}

Practically, finding the particular $(j,2k)$ which realizes Theorem \ref{thm:se} is not always feasible, since it is more complicated than ruling out the existence of a $(j,2k)$ $E$-flex for some particular $j$ and $k$.
Our
results Theorems \ref{thm:main} and \ref{thm:main2} below
show that, in certain situations, the nonexistence of certain flexes does lead to a conclusive statement about the growth order.  
These results correspond, respectively, to the classical notions
of first  and second-order rigidity
and  Theorem \ref{thm:alex}.

\subsection{Flexes and E-flexes}
The next step is to relate $(j,2k)$ E-flexes of a framework to $(j,k)$-flexes,
which connects the energy formalism to rigidity.  Along these lines,
Salerno \cite{salerno} proved the following.
\begin{theorem}[{\cite[Section 2.2]{salerno}}]
  \label{thm:sal}
  Let $(G,\p)$ be a framework and let $E(\q)$ be the harmonic
  spring energy.
  A trajectory $\p(t)$ is a  $(j,k)$-flex of $(G,\p)$ if and only if
  it is a $(j,2k)$ E-flex of $(G,\p)$.
\end{theorem}

We generalize Salerno's result to arbitrary stiff-bar energies.

\begin{theorem}
  \label{thm:coin}
  Let $E$ be a stiff-bar energy at $\p$.
  A trajectory $\p(t)$ is a  $(j,k)$-flex of $(G,\p)$ if and only if
  it is a $(j,2k)$ E-flex of $(G,\p)$.
  Furthermore, $\p(t)$ is a $(j,2k)$ E-flex if and only if it is a
  $(j,2k+1)$ E-flex.
\end{theorem}

The proof of this theorem is based on Taylor expansions of
$l_{ij}(\p(t))$ and $E_{ij}(l_{ij}(\p(t)))$. A key tool is Faà di
Bruno's formula, which extends the chain rule to higher derivatives
\cite{faa1,faa2}. Given  univariate analytic functions $f,g$,
Fa\`{a} di Bruno's formula says that
\begin{equation}\label{eq:faa}
  \frac{1}{n!}\frac{d^n}{dt^n}f(g(t))\Big|_{t=0} = \sum_{\mathbf j}
  \frac{1}{j_1!1!^{j_1}j_2!2!^{j_2}\cdots j_n!
  n!^{j_n}}f^{(j_1+\cdots +j_n)}(g')^{j_1}(g'')^{j_2}\cdots (g^{(n)})^{j_{n}},
\end{equation}
where $f$ and its derivatives are evaluated at $g(0)$ and $g$ and
its derivatives are evaluated at $t=0$.
The sum is over all sequences of integers $\mathbf j = (j_1,\ldots,
j_{n})$ satisfying
\begin{equation}\label{eq:faaj}
  1\cdot j_1 + 2\cdot j_2 + \cdots + n\cdot j_{n} = n.
\end{equation}

\begin{remark}
\label{rem:writtenOut}
  Written out explicitly, the first few derivatives are:
  \begin{align*}
    \frac{d}{dt}f(g(t))\big|_{t=0} &=   f'g'\\
    \frac{d^2}{dt^2}f(g(t))\big|_{t=0}&=  f'g''+ f''(g')^2 \\
    \frac{d^3}{dt^3}f(g(t))\big|_{t=0}&= f'g'''+ 3f''g'g'' + f'''(g')^3 \\
    \frac{d^4}{dt^4}f(g(t))\big|_{t=0}&= f'g^{(4)}+
    f''\left(4g'g'''+3(g'')^2\right) + f'''6(g')^2g'' + f^{(4)}(g')^4 \\
    \frac{d^5}{dt^5}f(g(t))\big|_{t=0}&= f'g^{(5)}+
    f''\Big(5g'g^{(4)} + 10g''g'''\Big) + f'''\Big( 10(g')^2g''' +
    15g'(g'')^2\Big)
    \\&\qquad + f^{(4)}10(g')^3g'' +f^{(5)}(g')^5 \\
    \frac{d^6}{dt^6}f(g(t))\big|_{t=0}&= f'g^{(6)}
    +f''\Big(6g'g^{(5)} + 15g''g^{(4)} + 10(g''')^2 \Big)
    \\& \qquad
    + f'''\Big( 15(g')^2g^{(4)} + 60g'g''g'''+ 15(g'')^3  \Big)
    + f^{(4)}\left(20(g')^3g''' + 45(g')^2(g'')^2 \right)
    \\& \qquad
    + f^{(5)}15(g')^4g'' + f^{(6)}(g')^6
  \end{align*}
\end{remark}

We now state two lemmas regarding Taylor expansions of arbitrary
functions. The first lemma is used to prove that a $(j,k)$
$\m$-flex is equivalent to a $(j,k)$ $\l$-flex (Remark \ref{rem:lorm}).

\begin{lemma}\label{lem:vanish1}
  Given  univariate analytic functions $f$ and $g$,
  if $g$ is $k$-vanishing then $f\circ g$ is $k$-vanishing.
  If $g$ is $k$-active and $f'|_{g(0)} \neq 0$, then $f\circ g$ is $k$-active.
\end{lemma}

\begin{proof}
  The first statement follows directly from \eqref{eq:faa}, since
  each term in the $j$-th derivative of $f\circ g$ involves
  products of some of the first $j$ derivatives of $g$, which are
  assumed to vanish.

  For the second statement, notice that if $g$ is $k$-active then
  it is $k-1$ vanishing, so $f\circ g$ is too. The $k$th derivative of $f\circ g$ is
  \[
    \frac{d^k}{dt^k}f(g(t))\Big|_{t=0} =
    f'g^{(k)} + R,
  \]
  where the terms contained in $R$ each involve a product of at
  least one derivative of the form $g^{(j)}$ with $j<k$. Since
  these terms vanish, so does $R$. Since $f'\neq 0$ and
  $g^{(k)}\neq 0$, we have that $\frac{d^k}{dt^k}f(g(t))\Big|_{t=0}
  \neq 0$, so $f\circ g$ is $k$-active.
\end{proof}

The second lemma is a key step in the proof of Theorem \ref{thm:coin}.
It will be applied to functions $g=l_{ij}$, $f = E_{ij}$, along a
trajectory $\p(t)$.

\begin{lemma}\label{lem:vanish2}
  Suppose $f'|_{g(0)} = 0$, but $f''|_{g(0)} \neq 0$.
  Then 
  (i) If $g$ is $k$-vanishing, then $f\circ g$ is $2k+1$-vanishing.
  (ii) If $g$ is $k-1$-vanishing, then
  \begin{equation}\label{eq:fg2k}
    \frac{1}{(2k)!}\frac{d^{2k}}{dt^{2k}}f\circ g\Big|_{t=0} =
    \frac{1}{2(k!)^2}(f'')(g^{(k)})^2.
  \end{equation}
\end{lemma}

\begin{remark}
  It is additionally possible to show under the same assumptions that (iii) If $f\circ g$ is
  $2k$-vanishing, then $g$ is $k$-vanishing. As this statement is a
  special case of Theorem \ref{thm:coin} we omit its proof. Thus,
  Lemma \ref{lem:vanish2} and (iii) imply the following:
  \begin{itemize}
    \item If $g$ is $k$-active, then $f\circ g$ is $2k$-active.
    \item If $f\circ g$ is $2k$-active, then $g$ is $k$-active.
  \end{itemize}
\end{remark}

\begin{proof}
  (i) Notice that for  $n\leq 2k+1$, each term in the sum
  \eqref{eq:faa} to compute $(f\circ g)^{(n)}|_{t=0}$ must involve
  at least one derivative $g^{(j)}$ with $j\leq k$. This follows by
  considering combinations of integers satisfying \eqref{eq:faaj},
  and observing that $f'=0$ so the term involving $g^{(2k+1)}$
  vanishes. Hence, each term in the sum is zero by assumption, so
  $(f\circ g)^{(n)}|_{t=0}=0$.

  (ii) Suppose $g$ is $k-1$-vanishing, and consider the terms in
  the sum \eqref{eq:faa} for $(f\circ g)^{(2k)}|_{t=0}$. All terms
  contain at least one derivative $g^{(j)}$ with $j\leq k-1$ except
  one term, leading to expression \eqref{eq:fg2k}.
\end{proof}

\begin{lemma}
\label{lem:formula}
Let $E$ be a stiff-bar energy 
at $\p$.
Suppose that $\p(t)$ is a $(j,k-1)$
flex. Then 
  \ba
  \frac{1}{(2k)!}\frac{d^{2k}}{dt^{2k}}E(\p(t))
  \Big|_{t=0} = \frac{1}{2(k!)^2}\sum_{ij}(E_{ij}'')
  (l_{ij}^{(k)})^2
  \ea
  where $l^{(k)}_{ij} := \frac{d^k} {dt^k}l_{ij}(\p(t))\big|_{t=0}$.
\end{lemma}
\begin{proof}
 Apply Lemma \ref{lem:vanish2} (ii)  to
  each $g=l_{ij}(\p(t))$, which is $k-1$-vanishing.
\end{proof}

\begin{proof}[Proof of Theorem \ref{thm:coin}.]
  We show the following are equivalent:
  \begin{enumerate}[(i)]
    \item $\p(t)$ is a  $(j,k)$ flex.
    \item $\p(t)$ is a $(j,2k+1)$ E-flex.
    \item $\p(t)$ is a $(j,2k)$ E-flex.
  \end{enumerate}
  (i)$\Rightarrow$(ii): Apply Lemma \ref{lem:vanish2}  to each
  $f=E_{ij}(l)$ with $g=l_{ij}(\p(t))$. Specifically:
  if $\p(t)$ is a $(j,k)$ flex, then each $l_{ij}(\p(t))$ is
  $k$-vanishing, so by Lemma \ref{lem:vanish2},
  $E_{ij}(l_{ij}(\p(t)))$
  is $2k+1$-vanishing. Therefore
  $E(\p(t))$
  is also $2k+1$-vanishing.

  (ii)$\Rightarrow$(iii) follows by definition.

  (iii)$\Rightarrow$(i): We prove this by induction. Consider the
  base case, $k=1$, and suppose that $\p$ is a $(j,2)$ E-flex, so
  that $\frac{d^2}{dt^2}E(\p(t))|_{t=0} =0$. Since  $E$ is stiff
  bar, we have, taking two time derivatives and evaluating at $t=0$,
  \[
    \frac{d^2}{dt^2}E(\p(t))\Big|_{t=0} = \sum_{ij}
    (E_{ij}'')(l_{ij}')^2 = 0
  \]
  where $E''_{ij} := \frac{d^2}{dl^2}E_{ij}(l)\big|_{l=d_{ij}} $,
  $l_{ij}' := \frac{d}{dt}l_{ij}(\p(t))\big|_{t=0}$.

  Since $E_{ij}'' > 0$, the above can only hold if $l_{ij}' = 0$,
  which implies that $p$ is a $(j,1)$ $\l$-flex and
  also a $(j,1)$ $\m$-flex (see Remark~\ref{rem:lorm}).

  Now suppose the statement holds for $k-1$, and suppose that
  $\p(t)$ is a $(j,2k)$ E-flex. Then by definition $\p(t)$ is
  also a $(j,2(k-1))$ E-flex, so by the induction hypothesis it
  is a $(j,k-1)$-flex.
  From Lemma~\ref{lem:formula}
  we have
  \ba
  \frac{1}{(2k)!}\frac{d^{2k}}{dt^{2k}}E(\p(t))
  \Big|_{t=0} = \frac{1}{2(k!)^2}\sum_{ij}(E_{ij}'')
  (l_{ij}^{(k)})^2
  \ea
  where $l^{(k)}_{ij} := \frac{d^k} {dt^k}l_{ij}(\p(t))\big|_{t=0}$.
  Since each $E_{ij}''>0$, the above can only equal 0 if each
  $l_{ij}^{(k)}=0$, hence, $\p(t)$ is a $(j,k)$ $\l$-flex and
  also a $(j,k)$ $\m$-flex (see Remark~\ref{rem:lorm}).
\end{proof}

\begin{remark}
  Because Theorem \ref{thm:coin} holds for any stiff-bar energy,
  and part (i)
  does not depend on $E$, rigidity statements that can be
  formulated in terms
  of (non-)existence of $(j,k)$ $E$-flexes either hold for every
  stiff-bar energy
  $E$ or for none of them.  For example, Lemma \ref{lem:eflex} and
  Theorem \ref{thm:coin} imply that, if $(G,\p)$ has a $(j,k)$-flex then,
  any stiff-bar energy $E$ at $\p$ grows sometimes-$s$-slowly for
  $s = \frac{2k+2}{j}$.
\end{remark}

\subsection{Related work}
Salerno~\cite{salerno}
(see also~\cite{vass})
describes an energy-based
rigidity analysis
that tests for
$1$-active trajectories that
are $(1,2k)$ E-flexes for higher and higher values of $k$.
The approach stops when $k$ cannot be increased.
In the case that
$\dim(K)=1$, for each $k$, the test can be
done efficiently using linear algebra.
When $\dim(K) >1$ the method becomes more
complicated and relies on numerical optimization.

In light of Lemma~\ref{lem:eflex},
Salerno's method is
certifying that $E$  grows sometimes-$s$-slowly where
$s=\frac{2k+2}{j}$, for higher and higher values of $k$.
But the halting-level of this  method does not show that
$E$ is growing always-$k$-quickly for some value of $k$.
In particular,
it only considers  $1$-active trajectories.
When  $(G,\p)$ is a cusp mechanism, this method
will halt at some finite $k$, even though $E$ grows
sometimes-$\infty$-slowly.

Responding to the existence of cusp mechanisms,
Garcea et al~\cite{gar} modified Salerno's method
for the case where $\dim(K)>1$, and considered
$(j,2k)$ E-flexes
over
various values of $j$.
This
gives  them the possibility of finding
higher values of  $s$  for sometimes-slow growth of $E$ than will
be found by Salerno's method.
But they say~\cite[Section 3.6.2]{gar},
that, in the general case,
this approach will not be able to
establish that $E$ grows always-$s$-quickly for any $s$,
as  they can
only explore up to some finite $j$.

Notably, like, Salerno,
in the $\dim(K)=1$ case,
they only consider
$1$-active trajectories.
(A  $1$-vanishing trajectory would instead require a
Puiseux series in their Equation (35)).
No explicit mathematical
justification is given for this, but
in light of the results of our paper, this
is provably a correct method for determining
the rigidity order of a framework.
In this paper
we  prove,
from the point of view of critical point analysis,
that in the case of
$\dim(K)=1$,
only $1$-active trajectories need to be
explored
in order to determine the tight growth order of $E$.

Our paper is inspired by earlier work in~\cite{hk}.
In that paper, they show that second-order rigidity is equivalent
to $\p$ being a strict (global) minimum of a certain
4th order approximation of a specific energy function.
But they do not prove that the strict minimality of this
4th order function implies strict (local) minimality of
the full energy function. So~\cite{hk}
does not provide  a new
proof that second-order rigidity implies rigidity.

Another recent paper~\cite{varda}
discusses some connections between rigidity and energy.  In particular
their Appendix 2  provides an alternative proof
of Salerno's Theorem~\ref{thm:sal} for
$k=2$.

\section{Definition of rigidity order}
\label{sec:orders}

Given this setup, we define the rigidity order of a framework using
the tight growth order for stiff bar energies. We must first establish that
the tight growth order is independent of the particular energy in question.
We do this by relating the tight order to the existence of $(j,k)$-flexes.

\begin{theorem}\label{thm: order jk-flex}
  Let $(G,\p)$ be a rigid framework.
  Then there exists an $s\in \mathbb Q$ such that every stiff bar
  energy for $(G,\p)$
  has tight growth order of $s$ at $\p$.
  Furthermore,  $s$ may be computed as
  \begin{equation}\label{sflex}
    \frac{s}{2} =  \max \left\{ \frac{k+1}{j} : \text{$(G,\p)$
    has a $(j,k)$-flex}\right\}.
  \end{equation}
\end{theorem}

\begin{proof}
  Let $E$ be some particular stiff bar energy for $(G,\p)$.
  Then  $E$ has tight growth order of $s_E\in \mathbb Q$, as
  established in Theorem \ref{thm: netto} and Remark
  \ref{remark:netto}.  By hypothesis, $E$ grows
  always $s_E$-quickly at $\p$.
  By Theorem \ref{thm:se}, $s_E$ may be computed as the maximum of $(2k+2)/j$ over all $(j,2k)$ E-flexes. By Theorem \ref{thm:coin}, a $(j,2k)$ E-flex is a $(j,k)$-flex, and conversely. Hence, $s_E$ may be computed as the maximum of $(2k+2)/j$ over all $(j,k)$-flexes. 

  We have established that \eqref{sflex} holds for a particular
  energy $E$ with tight growth order $s_E$. Since the right-hand
  side is independent of $E$, so is $s_E$.
  


\end{proof}

We thus propose the following
definition for
the rigidity order
of a framework.
\begin{mdframed}
  {
    \begin{definition}\label{defn:order}
      Let $(G,\p)$ be a rigid framework.
      The \defn{rigidity order}
      of $(G,\p)$ is defined as
      $\nu=s/2\in \mathbb Q$, where $s$ is the tight growth order
      at $\p$
      of any stiff bar energy associated with $(G,\p)$.
      (The choice of energy is not material due to  Theorem
      \ref{thm: order jk-flex}.)
      In light of Theorem \ref{thm: order jk-flex},
      $\nu=s/2$ can also be
      computed using Equation \eqref{sflex}.
    \end{definition}
  }
\end{mdframed}

As we show next in this paper, in some cases, one can determine the
rigidity order by looking for the lack of specific
flexes. In the general case, the tight growth order
could possibly be explored using symbolic algebra
methods~\cite{netto,pham}.
We note that nothing,
a-priori, rules out fractional values for this
rigidity order. An example that seems to have a
fractional order appears in~\cite[Example 7]{narwal}.

\subsection{Related work}
\label{sec:prvOrder}
Our work is inspired and informed by related ideas
from the literature.
Garcea, et al. \cite{gar} say that a framework
$(G,\p)$ with a $(j,k)$ flex is
an ``infinitesimal mechanism of order $k/j$''.
But as described in Remark~\ref{rem:irrFlex},
a $(1,1)$ flex can be reparameterized
to be a $(2,3)$ or a $(3,5)$ flex,
each with a different value of $k/j$.
Note that reparameterization
does not cause any difficulties
with the quantity $\frac{k+1}{j}$, or 
have any effect on the
growth order.

In an attempt to work around this  issue,
Stachel~\cite{stachel1999,stachel} defines
a $(j,k)$ flex to be ``irreducible'' if  it is not  a
reparameterization of a $(j',k')$ flex
where $j'<j$.
He then says that a framework
with an irreducible $(j,k)$ flex
has a ``fractional order infinitesimal flexibility
of $k/j$''.
But problems can still arise even
assuming irreducibility.
As described in~\cite{narwal}
Stachel later described
a framework such that,
for each $j$, has an irreducible
$(j,3j-1)$ flex, but no
$(j,3j)$ flex. For this framework,
the fraction
$(3j-1)/j$ does not achieve a maximum value
over all $j$.
Note that, from our point of view
all of these $(j,3j-1)$ flexes correspond to a
$\frac{k+1}{j}$ value of $3$.

In response to this example,
in a very recent paper~\cite{narwal2},
Nawratil proposes a further change to Stachel's definition by restricting
the trajectories $\p(t)$ to those that (quoting directly)
\begin{quote}
  {\em can be
    extended to a minimal parametrization of a branch of order $j$ of an
    algebraic curve, which corresponds to a one-dimensional
    irreducible component
    of a variety determined by an ideal, whose generators are contained in
    the linear
    family of quadrics spanned by [the
  squared-length edge constraints].}
\end{quote}
Nawratil indicates in \cite{narwal2}
that,
there are well defined values $j$ and $k$ associated with the
``highest real flex''
of the framework, subject to this
limitation.

Generally speaking,
the focus in these works on a flexibility-order
instead of a rigidity-order
leads to the fractional
value $k/j$ which
does not behave  as simply as
the fractional
value $\frac{k+1}{j}$. In our definition no
specific maximum values for $j$ and $k$ are defined,
or considered important. All that matters for us is
the order of energy growth.

Instead of studying flexibilty and rigidity
directly in the language of $(j,k)$ flexes,
Tachi~\cite[Section 5]{tachi} defines
$(G,\p)$ to be ``$s$-order infinitesimally
flexible'' if there is a
$C^\infty$ trajectory $\p(t)$
such that
\ba
|\p(t)_i - \p(t)_j|^2 = |\p_i - \p_j|^2 + o(|\p(t) - \p|^s),
\ea
for all edges $ij\in E(G)$.
Relating this back to flexes,
it can be shown that
if $(G,\p)$ has
a $(j,k)$ flex, then
\ba
|\p(t)_i - \p(t)_j|^2 - |\p_i - \p_j|^2 \le c|\p(t) - \p|^{\frac{k+1}{j}}
\ea
for some $c>0$.
Thus for any $\eps>0$, $(G,\p)$
would be $(\frac{k+1}{j}-\eps)$-order
infinitesimally flexible under Tachi's definition.
In particular, it can be shown that
if we define $s^*$ to be the
supremum over $s$ such $(G,\p)$ is
$s$-order infinitesimally flexible
according to Tachi's definition,
then $s^*$ is
equal to the rigidity order as we have defined it.  By looking at
rigidity order instead of flexibility
order, we avoid the need for taking
a supremum.

In~\cite{narwal}, Nawratil defines
a  very different
notion  of a flexibility-order
using algebraic multiplicity. But, as he points
out, since his definition is based on complex numbers,
it cannot assign any finite order to a
rigid
framework of a graph that is generically
flexible. This issue is not a problem for his
more recent definition in~\cite{narwal2}.

\section{Main results on rigidity tests}
The characterization of rigidity order using
$(j,k)$ flexes requires looking over all
positive integer values for $j$.
But in some cases, we can actually fix $j=1$ and
still obtain tests for rigidity order,
just looking at $(1,k)$ flexes.

We can now summarize
the main results of our paper regarding
rigidity tests and
rigidity orders

\begin{theorem}
  \label{thm:main0}
  Let $(G,\p)$ be a framework
  that  has  no $(1,1)$-flex.
  Then the rigidity  order of $(G,\p)$ is $1$.
  Moreover, if it has a $(1,1)$-flex, then it is either flexible
  or it is rigid with a rigidity order of at least $2$.
\end{theorem}

\begin{theorem}
  \label{thm:main}
  Let $(G,\p)$ be a framework
  that  has a $(1,1)$-flex but no
  $(1,2)$-flex. Then the rigidity  order of $(G,\p)$ is $2$.
  Moreover, if it has a $(1,2)$-flex, then it is either flexible
  or it is rigid with a rigidity order of at least $3$.
\end{theorem}
\begin{theorem}
  \label{thm:main2}
  Let $(G,\p)$ be a framework
  with $\dim(K)=1$.
  Suppose that for some $k$,
  it has a $(1,k-1)$ flex but no
  $(1,k)$ flex. Then its rigidity order is $k$.
  Moreover, if it has a $(1,k)$ flex, then it is either flexible
  or it is rigid with a rigidity order of at least $k+1$.
\end{theorem}

\vspace{.1in}
\begin{mdframed}
  {
    Theorem~\ref{thm:main2} tells us that
    when $(G,\p)$ is rigid and $\dim(K)=1$,
    the algorithm from Theorem \ref{thm:eff} can be used to
    compute its rigidity order. We recall that
    for fixed $k$, the
    algorithm is efficient.
    (Since $k$ is not  bounded
      by a polynomial in $n$,
    the algorithm, while of practical use, is not ``polynomial time''.)
  }
\end{mdframed}

\begin{remark}
  We see from these theorems that a framework
  cannot have a non-integer rigidity order below
  $3$, and when $\dim(K)=1$ it must have an integer
  rigidity order.
  Generalizing beyond the context of rigidity theory,
  at a strict local minimum
  of any analytic function $E$, the tight growth order
  must be an integer if it is below $6$. When the
  nullity of the Hessian of $E$ is $1$ at a critical point $\p$, the
  tight growth order of $E$ at $\p$ must be an integer.
\end{remark}

\section{Critical point analysis and a 4th derivative test}
Our study of rigidity in this paper is done using
an energy function $E(\q)$, for which we are interested in showing
that $\p$ is a strict local minimum. When the Hessian of $E$ is
positive definite at $\p$ then the second derivative test
certifies a strict local minimum. However, higher-order rigidity
requires certifying a local minimum when the Hessian of $E$
is degenerate.

We develop such higher-order derivative tests, based on extending
the Taylor-expansions considered in the second-derivative test.
The second-derivative test can be viewed as a Taylor-expansion of
a function along \emph{lines} through the origin. For a
fourth-derivative test, we consider a Taylor-expansion along a
suitable family of \emph{parabolas} through the origin.
Higher-order derivative tests consider families of subsequently
higher-degree polynomials, though we discuss how these tests only
apply in certain restricted settings. 

In the general setup, we will have a sufficiently smooth
function $f$ of $n$ scalar variables.
We will assume without loss of generality that $f$ has a critical point
at the origin, at which $f(0)=0$.

\subsection{The $2$nd derivative test}
Let $f$ be a $C^2$ function 
over $\RR^N$
with a critical point at the origin.
Let $H$ be the Hessian matrix
of $f$ at $0$.
The \defn{second derivative test} from first year
calculus tells us  \cite{Nocedal}:
\begin{enumerate}[\rm (a)]
  \item If $H$ is positive definite (PD),
    then $0$ is a strict local minimum
    of $f$.
  \item If $H$ is negative definite (ND), then $0$ is a strict
    local maximum of $f$.
  \item If $H$ is indefinite (ID), then $0$ is a saddle of $f$.
  \item Otherwise, $H$ is either positive semi-definite (PSD) or
    negative semi-definite (NSD)
    and the test is inconclusive.
\end{enumerate}
In case (d), we say that the critical point is \defn{degenerate}.  If $H$ is
non-zero in case (d) there is some information: if it is non-zero
and PSD, then the
origin cannot be a strict or weak local maximum; similarly if it
is  non-zero and
NSD, then the origin cannot be a strict or weak local minimum.
To go further requires some kind of higher order derivative test
(assuming that $f$ is sufficiently smooth).

\subsubsection{The problem with extending the 2nd derivative test}
In the univariate setting, devising such a higher order derivative test
is straightforward.  A degenerate critical point has a zero
second derivative.
Thus a non-zero third derivative certifies a saddle. If the third
derivative vanishes, a positive fourth derivative certifies a
strict local minimum while
a negative fourth derivative certifies a
strict local maximum, etc.

The multivariate setting is much more subtle.
One way to understand the second derivative test is through the multivariate
Taylor's theorem. When the Hessian is
PD, then the second-order approximation of $f$ is increasing
at second-order,
and the magnitude of the remainder term grows at order higher
than two.
As a result,  the remainder term
cannot overcome the second-order growth in a
neighborhood of the origin.

What happens when we look at, say, the fourth order approximation
of a function $f$? We will see next that an inhomogeneous
fourth order approximation might grow in all directions, but
this growth may actually be slower than fourth order! When that happens,
its behavior can still be dominated by the remainder term.

\begin{figure}
    \includegraphics[width=0.48\linewidth]{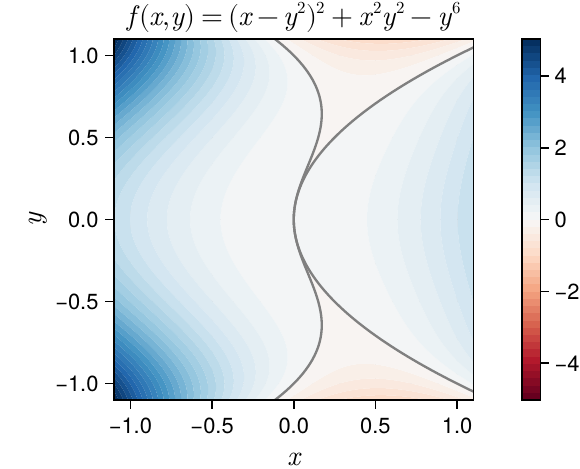}
    \includegraphics[width=0.48\linewidth]{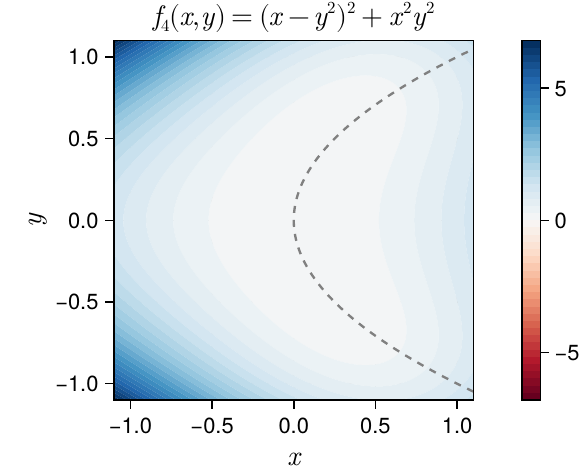}
    \caption{Left: A sixth order polynomial that has a saddle at the origin. The zero set is shown as solid curves.  Right: Its fourth degree Taylor approximation has a strict minimum at the origin. This fourth degree polynomial only grows at 6th order 
    at the origin along the shown dashed curve.}\label{fig:ex6.1}
\end{figure}

\begin{example}
Consider the function
\bna
\label{eq:rud}
f(x,y):=(x-y^2)^2 + x^2y^2 - y^6.
\ena
This function  has a saddle
at the origin (Figure  \ref{fig:ex6.1}(left)).
To see this directly, one can first show that $f$ factors as
  $f(x,y)=(x-y^2)g(x,y)$ with $g(x,y)=x-y^2+y^2(x+y^2)$. Choose a
  point on the parabola $x=y^2$ where $g(x,y)\neq 0$ and which is
  arbitrarily close to the origin. Then change $x$ by a small
  amount in either direction, to show that $f$ must cross zero near
this point.

Now consider an approximation to this function which keeps all
terms up to 4th order in its Taylor expansion (Figure  \ref{fig:ex6.1}(right)):
\ba
f_4(x,y):=(x-y^2)^2 + x^2y^2 .
\ea
This approximation has a strict
local minimum at $0$: The first term is positive except
on the parabola $x=y^2$, and the
second term increases the function
everywhere except at the origin.
Therefore, simply considering the fourth-order Taylor expansion
of a multivariate function is not sufficient to determine its
behavior near a degenerate critical point.

The problem is that
$f_4(x,y)$  actually grows slower than
fourth order in $|(x,y)|$ along some curve.
For example, along the trajectory $(x(t),y(t))=(t^2,t)$, we have
$f_4(x(t),y(t))=t^6$ (Figure  \ref{fig:ex6.1}(right)). Thus,
adding a higher order term to $f_4$
can change the nature of the critical point.
(This behavior does not
  arise with the second-derivative test,
  where at a critical point, the
second-order approximation is homogeneous.)
\end{example}

\vspace{.1in}
An alternative way to understand the
multivariate
second derivative test is by thinking
of it as an analysis of the family of
all univariate slices arising from restricting
$f$ to  lines through the origin.
When the Hessian is, say, PD,
each univariate slice will have a zero first derivative and
positive second derivative,
implying that the function increases in every direction.
From this perspective, as well, the multivariate higher order setting
presents additional subtleties.

\begin{figure}
\centering
    \includegraphics[width=0.48\linewidth]{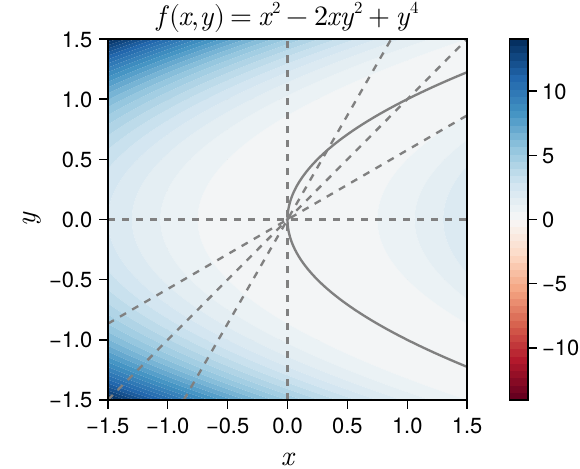}
    \caption{A fourth degree polynomial that has a non-strict local minimum at 
    the origin. The zero locus is shown as a solid curve. Notably, this function has a strict (univariate) local minimum along any linear trajectory
    through the origin (dashed lines).
    }\label{fig:ex6.2}
\end{figure}

\begin{example}
Consider the function
\ba f(x,y):=(x-y^2)^2=x^2-2xy^2+y^4
\ea
The origin is a degenerate critical point. The Hessian is PSD
with the $y$-axis
as its kernel direction. One might try to
analyze the critical point
by looking slice-by-slice
at all lines through the origin (Figure  \ref{fig:ex6.2}).
Doing this we find that $f$ has a strict local minimum in the
$y$ direction (due to the  $y^4$ term). We also find that
$f$ is a strict local minimum in any slice that is not the
$y$ direction (due to the dominating $x^2$ term).
But we would be wrong to
conclude that $f$ has a strict local minimum in two dimensions at
the origin. In fact $f$ is $0$ at any point on the parabola
$x=y^2$, and so the origin is actually a non-strict local minimum!

What went wrong in this reasoning?
As we look at lines
that are
getting closer and closer to the $y$-axis, we are finding a
univariate slices of $f$ that are staying positive, as we leave the
origin, but only for  shorter and shorter distances from the origin (Figure  \ref{fig:ex6.2}).
In particular,
we are unable
to obtain a uniform bound for the radius of some punctured disk
in 2D about the origin where
$f$ stays positive.
At the root of this problem is that as we look
at different slices, the leading degree of the
univariate polynomial in the slice changes,
from second degree generically, to fourth degree when we get
exactly to the $y$-axis.
If we had a uniform leading degree
for all of our slices, we would have obtained our needed uniform
bound.  In the next section
we will pursue this idea further and see
how we can come up with higher order derivative tests,
by using curves instead of lines.
\end{example}

\subsection{Indicative trajectories}
\label{sec:ind}

As just described,
we can interpret the second derivative tests as
probing $f$ using a set of ``test trajectories'' corresponding to
lines
through the origin.
The family of trajectories are parameterized
by a point $\a^{(1)}$ on the
unit sphere in $\RR^N$  as follows:
\ba
\x(t;\a^{(1)}) := \a^{(1)} t .
\ea
Given $\a^{(1)}$, we then consider
\ba
g(t) := f(\x(t;\a^{(1)})) .
\ea
If, for all $\a^{(1)}$,
$g(t)$ is positive for $t\neq 0$
(resp. negative/mixed-sign) up
through second-order in $t$,
then $f(\x)$ must have an strict local minimum
(resp. strict local maximum/saddle)
at the origin.
Otherwise,
the test is inconclusive.

We will show how to generalize
the second derivative test by identifying a suitable set of
test trajectories.  The key properties
that the set must have
are captured by the next
definition.
\begin{definition}
  \label{def:indic}
  Let $k$ be a positive integer.
  Let $f:\R^N\to \R$ be a
  $C^{2k+1}$  function
  with a critical point at the origin and with
  $f(0)=0$.
  Let $S$ be a set of parameter values.
  A \defn{family of test trajectories} $\{\x(t;\a)\}$
  is a set of trajectories
  defined over a shared interval
  $t \in [0,\eps]$
  (see Definition \ref{def: trajectory}), such that,
  for fixed $\a \in S$, the function $\x(t;\a)$
  is an analytic (in $t$), non-constant trajectory, with $\x(0,\a) = 0$.

  A family of test trajectories is \defn{indicative at order $2k$ for $f$}
  if it has the following properties:
  \begin{enumerate}[\rm ({I}1)]
    \item
      The parametrization set $S$ is compact.
    \item
      The function $\x(t;\a)$ admits a 
      $C^{2k+1}$ extension over an open neighborhood of $[0,\eps] \times S$.
    \item The family $\{\x(t;\a)\}$ of trajectories
      \defn{initially covers}
      a neighborhood of the origin: for every
      $0<\delta \le \eps$, there is a
      neighborhood $U$ of the origin, such that, for all $\x\in
      U$, there is an $\a \in S$
      and a
      $t\in [0,\delta]$, such that $\x = \x(t;\a)$.

    \item For all test trajectories $\x(t;\a)$, the Taylor
      expansion in $t$ at the origin
      is of the form
      \[
        f(\x(t;\a)) =  a_{2k}(\a) t^{2k} + o(t^{2k}).
      \]
      That is, all of the terms below order $2k$ are zero.
      In other words, for all $\a \in S$,
      $f(\x(t;\a))$ is $2k-1$-vanishing
      and also has a zero constant term.
      Thus $\x(t;\a)$ is a $(j,2k-1)$ f-flex
      for some $j$.
  \end{enumerate}
\end{definition}

The next proposition says that we can
use a family of indicative trajectories
to analyze a critical point. The ideas
behind this proposition can be seen
in~\cite[Proof of Theorem 1]{cushing}.

\begin{proposition}
  \label{prop:indic}
  Let $k$ be a positive integer.
  Suppose  $f(\x)$
  is a $C^{2k+1}$
  function
  with a critical point at the origin,
  and $f(0)=0$.
  Suppose we have
  family  $\{\x(t;\a)\}$ of test
  trajectories that is indicative at
  order $2k$ for $f$.
  \begin{enumerate}[\rm (a)]
    \item
      If for all parameter settings $\a$,
      $a_{2k}(\a)$ is positive,
      then there exists a $c > 0$ and a
      neighborhood $U$ of the origin
      such that for all $\x\in U$,
      $f(\x) \geq c|\x|^{2k}$.  In
      particular, $f$ has a strict local minimum
      at the origin.

    \item
      If for all parameter settings $\a$,
      $a_{2k}(\a)$ is negative,
      then there exists a $c > 0$ and a
      neighborhood $U$ of the origin
      such that for all $\x\in U$,
      $f(\x) \le -c|\x|^{2k}$.  In
      particular, $f$ has a strict local maximum
      at the origin.
    \item
      If  $a_{2k}(\a)$ is positive for
      for some parameter values $\a$, and negative for others,
      then the origin is a saddle point of $f$.
    \item Otherwise, $a_{2k}(\a)$ is either
      identically zero, or is
      sometimes positive and sometimes zero, or
      is  sometimes negative and sometimes zero.
      In any of these cases, the test is inconclusive.
  \end{enumerate}
\end{proposition}
\begin{proof}
  {\bf (a):}  Assume that $a_{2k}(\a) > 0$ for all
  parameters $\a$.
  The function $f$ is in $C^{2k+1}$ and
  \[
    a_{2k}(\a) =
    \left.\frac{1}{(2k)!}\frac{d^{2k}}{dt^{2k}}f(\x(t;\a))\right|_{t=0}.
  \]
  Because $f$ and $\x(t;\a)$ are $C^{2k + 1}$, $a_{2k}(\a)$
  is continuous
  in $\a$.  Hence, since $a_{2k}(\a)$ is assumed to be
  always positive, its minimum value, $a$, over the compact
  set $S$ is positive.
  Let us define
  \ba
  b = \sup_{x\in S, t_0\in [0,\eps]}
  \left|\frac{1}{(2k+1)!}\left.\frac{d^{2k+1}f(\x(t;\a))}
  {dt^{2k+1}}\right|_{t=t_0}\right|.
  \ea
  By compactness of $S\times [0,\eps]$ and smoothness of $f$ and $\x(t;\a)$,
  $b$ is finite.
  (This is the step where we rely on all the
    $2k+1$
  orders of continuity.)

  For any $\a$,  the univariate Taylor Theorem (with the
  remainder in Lagrange form)
  implies
  that
  \[
    f(\x(t;\a)) \ge  a_{2k}(\a)t^{2k} - bt^{2k+1} \ge at^{2k} - bt^{2k+1}.
  \]
  For small enough $\delta > 0$, the right hand side is at least
  $\frac{1}{2}at^{2k}$. This $\delta$ does not depend on $\a$.
  Hence, there is a $0<\delta < 1$,
  so that for $t\in[0,\delta]$
  and for all $\a$,
  \[
    f(\x(t;\a)) \ge \frac{a}{2} t^{2k}.
  \]
  Now we derive a uniform upper bound on
  $|\x(t;\a)|^2$ in terms of of $t$.  We
  define
  \[
    A =
    \sup_{t_0\in [0,\eps], \a\in S}
    \left\{
      \left.\frac{d^2}{dt^2} |\x(t;\a)|^2\right|_{t=t_0}
    \right\}.
  \]
  By smoothness of the norm and the parameterization, and
  compactness of $S$, along with the compactness
  of $[0,\eps]\times S$,  $A$ is finite.
  The first-order
  univariate Taylor
  expansion
  of
  $|\x(t;\a)|^2$
  is $0$, and using its
  (second-order) remainder in Lagrange
  form, we get, for all $\a\in S$, and
  $t \le \eps$,
  \[
    |\x(t;\a)|^2 \le At^2.
  \]
  Combining with our estimate on $f$, we get
  for $t \le \delta$ and all $\a\in S$:
  \[
    f(\x(t;\a)) \ge \frac{a}{2} t^{2k}
    \ge
    \frac{a}{2A^k}|\x(t;\a)|^{2k}.
  \]

  Finally, since we have a set of indicative trajectories, the
  initial coverage condition (I3) implies that there is a
  neighborhood $U$ of the origin, so that, for all $\x\in U$,
  there is a $t\in [0,\delta]$ and $\a\in S$,
  such that $\x = \x(t;\a)$.  We take this
  $U$ to be the one in the conclusion, since $\x\in U$
  implies that
  \[
    f(\x)
    \ge
    \frac{a}{2A^k}|\x|^{2k}.
  \]
  which proves
  (a).

  {\bf (b):} Apply (a) to $-f$.

  {\bf (c):} This case is easier, because we don't
  need a uniform bound over $S$.  Suppose that there
  is an $\a\in S$ such that $a_{2k}(\a) > 0$.  By
  the univariate Taylor theorem with Peano remainder
  \[
    f(\x(t;\a)) = a_{2k}(\a)t^{2k} + o(t^{2k}).
  \]
  For all sufficiently small $t > 0$, the right hand side has the same sign
  as $ a_{2k}(\a)t^{2k} > 0$; i.e., $ f(\x(t;\a)) > 0$.
  Similarly, if $a_{2k}(\a) < 0$, for all sufficiently small
  $t > 0$, $ f(\x(t;\a)) < 0$.  This shows that the origin
  is a saddle point of $f$.
\end{proof}

\subsection{A  $4$th derivative test}
\label{Sec:4test}

We now describe a derivative test that may be applied when the
Hessian of $f$ is degenerate. The key step is to construct a family
of  trajectories that are indicative at order $4$. Then,
Proposition \ref{prop:indic} may be used to determine the nature of
the critical point.
In the special case that $\dim(K)=1$,
this is the same as the test described by
Cushing~\cite{cushing}.

Let $f : \mathbb{R}^{n+m}\to \mathbb{R}$ be a $C^5$ function
of $n+m$ variables with a critical point at the origin.
Let us assume
\begin{itemize}
  \item[(A1)] The second derivative test is inconclusive; i.e.,
    the Hessian of $f$ at the origin is PSD of rank $n$ and nullity
    $m$, with $m\ge 1$.
\end{itemize}
In this section, we write a vector in $\RR^{n + m}$ as $(\x,\y)$,
where $\x = (x_1, \ldots, x_n)$ and $\y = (y_1, \ldots, y_m)$.  We further
assume:
\begin{itemize}
  \item[(A2)] The kernel of the Hessian $H_f(0)$ of $f$ at the origin is the
    subspace of vectors of the form $(0,\y)$.
\end{itemize}
If $f$ is any $C^5$ function $f : \mathbb{R}^{n+m}\to \mathbb{R}$,
satisfying (A1),
one can find a
basis $\{\v_1, \ldots, \v_n, \u_1, \ldots, \u_m\}$ of $\RR^{n+m}$
with the property that
the $\u_j$ span the kernel of $H_f(0)$.  After changing to this basis,
the Hessian will satisfy assumption (A2).

The assumption (A2) implies that in the Taylor expansion of $f$,
all degree two terms are of type $x_ix_j$.

We consider the following family
of test trajectories $\{(\x(t;\a^{(2)}),\y(t;\a^{(1)}))\}$, where
\bna
\label{eq:k2famXY}
{\x}(t;\a^{(2)}) &:=& \a^{(2)} t^2 \\
{\y}(t;\a^{(1)})&:=& \a^{(1)} t\label{eq:k2famXYb}
\ena
with $(\a^{(2)},\a^{(1)})$ on the unit sphere in
$\RR^{n+m}$, i.e. $|\a^{(2)}|^2+|\a^{(1)}|^2=1$.
Our motivation is to alter the order of the ``quickly growing''
directions $\x$, so that $f$ changes
at third order or higher in $t$.

We will be interested in the behavior of $f$ along these
trajectories. To this end, define a parameterized set of univariate functions
\begin{equation}\label{gt}
  g(t;\a^{(2)},\a^{(1)}) := f(\x(t;\a^{(2)}),\y(t;\a^{(1)})).
\end{equation}
Consider the Taylor expansion of $g$ in $t$ at $t=0$.
This is related to the Taylor expansion of $f$ in $\x,\y$. By our
construction of $\x,\y$, the only degree two terms in the Taylor
expansion of $f$ are of the form $x_ix_j$, which contribute to
terms of degree at least $4$
in the Taylor expansion of $g$.
The degree three terms in the Taylor expansion of $f$ are of the
form $y_iy_jy_k,x_iy_jy_k,x_ix_jy_k,x_ix_jx_k$, and have orders
$t^3,t^4,t^5,t^6$ respectively; the degree four term $y_iy_jy_ky_l$
has order $t^4$ and all other degree four terms are higher-order.
Hence, the Taylor expansion of $g$ has the form
\[
  g(t;\a^{(2)},\a^{(1)}) = a_3t^3 + a_4t^4 + o(t^4).
\]
The third-order coefficient $a_3(\a^{(2)},\a^{(1)})$ may be nonzero for
some parameter value  $\a^{(2)},\a^{(1)}$, only if there exists a term of
the form $y_iy_jy_k$ in the Taylor expansion of $f$.
When there is no such term, then $a_3(\a^{(2)},\a^{(1)})=0$ for all
parameter values $\a^{(2)},\a^{(1)}$, so all the terms in the Taylor
expansion of $g$ have degree at least
$4$.  With a little more work we can show that \eqref{eq:k2famXY}
is a set of indicative trajectories:
\begin{lemma}
  \label{lem:ind2a}
  Suppose (A1) and (A2) hold,
  and consider the third-order Taylor expansion of $f(\x,\y)$ in
  $\x,\y$ at the origin.
  If there is a term of the form $y_iy_jy_k$, or equivalently if
  $a_3(\a^{(2)},\a^{(1)})\neq 0$ for some $\a^{(2)},\a^{(1)}$, then the origin is a
  saddle point for $f$.
  Otherwise, the family of test trajectories
  $\{(\x(t;\a^{(2)}),\y(t;\a^{(1)}))\}$ is indicative  at order $4$ for $f$.
\end{lemma}

\begin{remark}\label{rem:minNoy3}
We see from this lemma that if we already know that $0$ is not a saddle, then
there can be no $y_iy_jy_k$ terms, and we do not have to do this check.
\end{remark}

\begin{remark}\label{remark:taylor4}
  For $g(t)$
  expanded through 4th order in $t$, the
  only terms of the Taylor series of $f$ that
  contribute at this order are of the type
  $x_i x_j$, $x_i y_j y_k$,
  and $y_iy_jy_iy_l$. In particular, we are not
  exploring the behavior of the full 4th order
  expansion of $f$, which can include terms of the
  form $x_ix_j x_k$,
  $x_ix_j y_k$,
  $x_i x_j x_k x_l$,
  $x_i x_j x_k y_l$,
  $x_i x_j y_k y_l$, and
  $x_i y_j y_k y_l$.
\end{remark}

\begin{proof}[Proof of Lemma \ref{lem:ind2a}]
  If there is a
  term of the form $y_iy_jy_k$ in the Taylor expansion for $f$,
  then the polarization identity,
  applied to the $3$rd derivative tensor, implies that there is a
  linear trajectory $\y(t)$ at the origin of the form
  $\y(t) = \a^{(1)} t$ such that the leading order term
  in the Taylor expansion of $f(0,\y(t))$
  in $t$ is of the form $a_3t^3$, with $a_3\neq 0$.  We
  see immediately that this implies that the origin
  is a saddle of $f$.

  For the statement about equivalence, if there are parameter
  values such that $a_3\neq 0$, then this can only come from
  terms of the form $y_iy_jy_k$ in the Taylor expansion for $f$.

  Now suppose there is no term of the form $y_iy_jy_k$ in the
  Taylor expansion for $f$. We check properties (I1)-(I4) of
  Definition \ref{def:indic}.

  (I1) The parameter space is a sphere, so it is compact.

  (I2) The parameterization is clearly smooth as a function
  $[0,\eps]\times S^{n+m-1}\to \RR^{n+m}$.

  (I4) follows by our assumption on the lack of terms of the form
  $y_iy_jy_k$.

  (I3) Now we show initial coverage for our family.
  To this end,
  for a fixed $s > 0$, define the ellipsoid $Q_s$
  over $\RR^{n+m}$, with coordinates $(\x,\y)$,
  to be the points that satisfy
  \ba
  |\x|^2+s^2 |\y|^2 -s^4 =0
  \ea
  The interior of $Q_s$ consists of the points
  $(\x,\y)$ where the left hand side is negative,
  while the exterior of $Q_s$ consists of the points
  $(\x,\y)$ where the left hand side is positive.

  Still keeping $s$ fixed, we define a  map from
  $\RR^{n+m}$ with coordinates $(\a^{(2)},\a^{(1)})$,
  to $\RR^{n+m}$ with coordinates $(\x,\y)$
  by
  \begin{align*}
  \x &:=\a^{(2)} s^2\\
  \y &:= \a^{(1)} s.
  \end{align*}
  The map is clearly a bijection.  The following
  calculation shows that under this  map, 
  $(\x,\y)\in Q_s$ if and only if $(\a^{(2)},\a^{(1)})$ is on the unit sphere $S^{n+m-1}$:
  \ba
  |\x|^2+s^2 |\y|^2 -s^4 &=&
  |\a^{(2)} s^2|^2+s^2 |\a^{(1)} s|^2 -s^4
  =s^4 (|\a^{(2)} |^2+|\a^{(1)} |^2-1).
  \ea

  Let $0<\delta <\eps $ be given.
  Let us consider the ellipsoid $Q_\delta$.
  For a small enough neighborhood $U$ of the origin,
  each point $(\x,\y) \in U$ is in the interior of $Q_\delta$.
  Meanwhile, for each point $(\x,\y) \in U$ and not equal to
  $0$, there must be a sufficiently small value
  $r>0$ so that $(\x,\y)$ is in the exterior of $Q_r$.
  Thus by continuity, there must be an intermediate
  value $r<t<\delta$ such that $(\x,\y)$ is on $Q_t$.
  Thus for this $t$, we have
  $\x=\a^{(2)} t^2$ and $\y=\a^{(1)} t$ for some
  $(\a^{(2)},\a^{(1)}) \in S^{n+m-1}$.
\end{proof}

\vspace{.1in}
\begin{mdframed}
  { Taken  together, Proposition~\ref{prop:indic} and
    Lemma~\ref{lem:ind2a} provides a novel
    and general $4$th derivative test for a
  function $f\in C^5(\R^{n+m})$ when (A1) holds. }
\end{mdframed}

To summarize the test:
\begin{enumerate}[(1)]
  \item Perform a linear change of variables
    so that the kernel
    of the Hessian of $f$ consists of the vectors supported only
    on their last
    $m$ coordinates, denoted by $y_i$, while  the first $n$ coordinates
    are denoted by $x_i$.
  \item Compute the Taylor expansion of $f$, and check for the
    existence of terms of the form $y_iy_jy_k$. If yes, then
    stop: 0 is a saddle. (Alternatively, Taylor-expand
      $g(t;\a^{(2)},\a^{(1)})$ in  \eqref{gt} directly, and stop if
    $a_3(\a^{(2)},\a^{(1)})\neq 0$ for some $(\a^{(2)},\a^{(1)})$.)
  \item Otherwise,
    apply Proposition \ref{prop:indic} to
    the family of trajectories from Equation
    \eqref{eq:k2famXY}. In particular, perform a Taylor-expansion
    of $g(t;\a^{(2)},\a^{(1)})$ defined in  \eqref{gt} and compute
    the fourth-order coefficients $a_4(\a^{(2)},\a^{(1)})$. Examine these
    over all parameter values such that $|\a^{(2)}|^2+|\a^{(1)}|^2=1$ to
    determine the nature of the critical point.
\end{enumerate}

In the general case, there may not be an efficient
algorithm to evaluate this test.

\vspace{.1in}

When this test is inconclusive, then there is a
parameter $(\a^{(2)},\a^{(1)})$ on the unit sphere in $\RR^{n+m}$
such that $a_4(\a^{(2)},\a^{(1)})=0$,
i.e.,
for the trajectory
\bna
\label{eq:ind2gen}
{\x}(t;\a^{(2)}) &:=& \a^{(2)} t^2\\
\label{eq:ind2gen2}
{\y}(t;\a^{(1)}) &:=& \a^{(1)} t
\ena
we have $f(\x(t;\a^{(2)}),\y(t;\a^{(1)}))=0$
through 4th order in $t$.

\begin{lemma}
  \label{lem:isactive}
  Suppose (A1) and (A2) hold.
  Suppose that the trajectory $(\x(t;\a^{(2)}),\y(t;\a^{(1)}))$ is
  such that $f(\x(t;\a^{(2)}),\y(t;\a^{(1)}))$ vanishes through fourth order
  in $t$.  Then $\a^{(1)}$ is non-zero.  In particular, if the 4th
  derivative test is inconclusive, then $(\x(t;\a^{(2)}),\y(t;\a^{(1)}))$
  is $1$-active.
\end{lemma}
\begin{proof}
  Suppose, for a contradiction, that $\a^{(1)} = 0$. Let $X$ be the coordinate
  subspace spanned by the $\x$-coordinates.  Since $\a^{(1)} = 0$, the trajectory
  $(\x(t;\a^{(2)}),\y(t;\a^{(1)})) = (\x(t;\a^{(2)}),0)$ is in $X$.  The Hessian
  of $f$ is definite on $X$, so the second derivative test will
  certify that $\tilde{f} := f|_X$ has a strict local minimum at the origin.
  Now define a trajectory $\tilde{\x}(\tau(t)) := \tau(t)\a^{(2)}$
  with $\tau(t) := t^2$.
  The trajectory $\tilde{\x}$ is linear in $\tau$, and so indicative for the
  second derivative test.  By assumption,
  $\tilde{f}(\tilde{\x}(\tau(t)))$,
  vanishes through fourth order in $t$ and so
  $\tilde{f}(\tilde{\x}(\tau))$,
  vanishes through second order in $\tau$,
  making the second derivative test
  inconclusive for $\tilde{f}$.  The resulting contradiction
  implies that the assumption $\a^{(1)} = 0$ was false.
\end{proof}

\begin{figure}
    \includegraphics[width=0.33\linewidth]{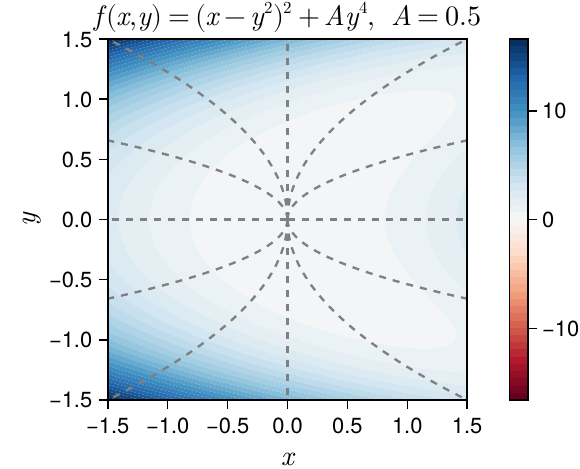}
    \includegraphics[width=0.33\linewidth]{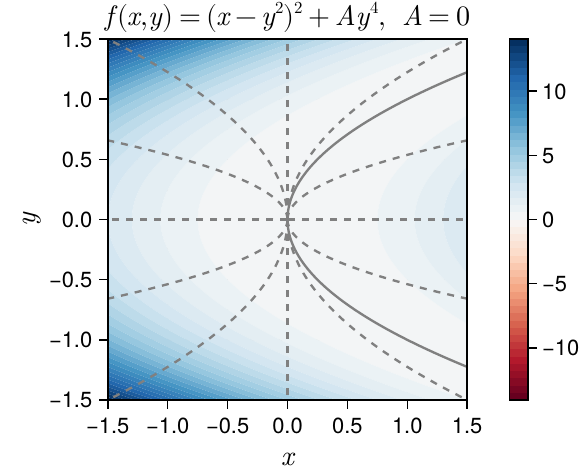}
    \includegraphics[width=0.33\linewidth]{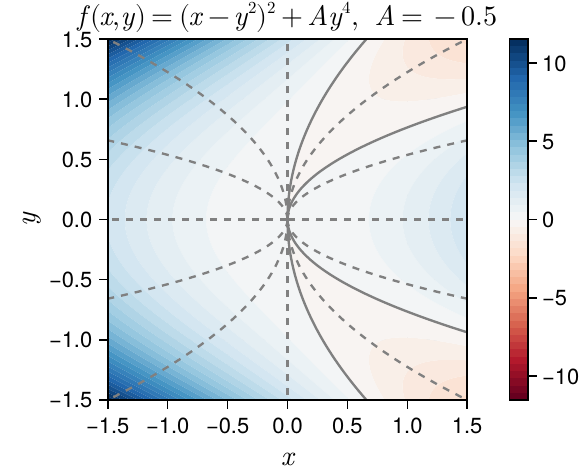}
    \caption{Three fourth degree polynomials. Zero sets are shown
    as solid curves. Indicative trajectories for the fourth derivative test are shown as dashed curves.
    Left: Strict local minimum at the origin. 
    Along each
    indicative trajectory, the energy has a positive, leading
    fourth degree coefficient, certifying the order of growth.
    Middle: Non-strict
    local minimum at the origin. 
        Along each
    indicative trajectory, the energy has a non-negative, leading
    fourth degree coefficient. Because of the zero coefficient along
    the solid parabola, the test is inconclusive.
    Right: Saddle at the origin.     Along the
    indicative trajectories, the energy has a mix of positive and negative, leading
    fourth degree coefficient, certifying saddle behavior. 
    }\label{fig:ex6.9}
\end{figure}

\begin{example}
Here is an example to illustrate the fourth-order critical point test.
Consider the fourth-degree polynomial  with Hessian
\[
  f(x,y) = (x-y^2)^2 + Ay^4, \qquad\quad
  \frac{1}{2}Hf\Big|_{(0,0)} =
  \begin{pmatrix} 1 & 0 \\ 0 & 0
  \end{pmatrix},
\]
where $A\in \R$ is some constant.
Then $f$ has a critical point at the origin, and the variable $y$
already spans the degenerate direction of the Hessian. Evaluating
$f$ along the indicative trajectories $x=a^{(2)}t^2, y=a^{(1)} t$ gives
\[
  f(a^{(2)}t^2,a^{(1)}t) = t^4\Big((a^{(2)}-{a^{(1)}}^2)^2 + A{a^{(1)}}^4\Big),
\]
so $a_4(a^{(2)},a^{(1)}) = (a^{(2)}-{a^{(1)}}^2)^2 + A{a^{(1)}}^4$.

When $A>0$ (Figure  \ref{fig:ex6.9}(left)), then $a_4(a^{(2)},a^{(1)})>0$ for all $|(a^{(2)},a^{(1)})|=1$, so
case (a) of Proposition \ref{prop:indic} tells us the origin is a
strict local minimum.

When $A=0$ (Figure  \ref{fig:ex6.9}(middle)), then $a_4(a^{(2)},a^{(1)})\geq 0$, with $a_4(a^{(2)},a^{(1)})=0$ at
two points on the unit sphere, $(a^{(2)},a^{(1)}) =
(\frac{\sqrt{5}-{1}}{2},\pm\sqrt{\frac{\sqrt{5}-{1}}{2}})$.
Case (d) of Proposition \ref{prop:indic} applies and the fourth
derivative test is inconclusive. (One can see that example
  \eqref{eq:rud} has the same expression for $a_4$ with $A=0$, so
it behaves identically to this case.)

When $A<0$ (Figure  \ref{fig:ex6.9}(right)), then $a_4(a^{(2)},a^{(1)})$ is sometimes positive (e.g. when
$a^{(1)}=0$) and sometimes negative (e.g. on the parabola
$a^{(2)}={a^{(1)}}^2$), so case (c) of Proposition \ref{prop:indic} tells
us the origin is a saddle point.

In light of Remark \ref{remark:taylor4}, when $A\neq 0$, we may
add  any polynomial of the form
$c_1x^3+c_2x^2y+c_3x^4+c_4x^3y+c_5x^2y^2+c_6xy^3$  to $f$,
without changing the nature of the critical point.
\end{example}

\subsection{What about a general 6th derivative test}
\label{sec:whatAbout}
Interestingly, if the 4th derivative test is inconclusive, it is
not clear how
to derive a  generally applicable 6th derivative test.

Suppose the 4th derivative test of Section~\ref{Sec:4test} is
inconclusive.
We can check if
there are any terms of the form
$y_i y_j y_k y_l y_m$ in the Taylor expansion of $f$.
If there are, then $f$ must be a
saddle.
But if there are no such terms, what then?

For sake of discussion, let us suppose
that there was a single trajectory
of the form of Equation (\ref{eq:ind2gen}),
up to a linear reparameterization in $t$,
governed by $(\a^{(2)},\a^{(1)})$.
One idea is to look at the family of
test trajectories
\ba
\x(t;\alpha,\a^{(3)}) &:=&  \alpha^2  \a^{(2)} t^2 + \a^{(3)} t^3\\
\y(t,\alpha) &:=&  \alpha \a^{(1)} t
\ea
with $\a^{(1)} \in \RR^m$ and $\a^{(2)} \in \RR^n$ fixed and with parameters $\alpha \in \RR$ and $\a^{(3)} \in \RR^n$
constrained with $|\alpha|^2+|\a^{(3)}|^2=1$.
But
when we project this family onto 
the space spanned by the $y$ coordinates,
we see that this projection will lie in
the line spanned by $\a^{(1)}$.
Therefore, the family will not cover a neighborhood of the
origin, so it is not an indicative family.

Interestingly, when the Hessian has nullity of
$1$, and there is only a single $y$ coordinate, then we \emph{will} get
neighborhood coverage from this family. Indeed, this kind of
family is at the heart of the
higher
order rigidity tests in Section~\ref{sec:higher}.

One might try to use a bigger family to cover
a neighborhood of the origin, but this can run
into other problems.
Consider the family
\ba
\x(t;\alpha,\a^{(3)}) &:=&   \alpha^2 \a^{(2)} t^2 + \a^{(3)} t^3\\
\y(t;\b^{(1)}) &:=&  \b^{(1)}  t
\ea
where we allow $\b^{(1)} \in \RR^m$ to be a free parameter along with $\a^{(3)}$
and $\alpha$,
constrained so that $|\alpha|^2+|\a^{(3)}|^2+|\b^{(1)}|^2=1$.
Then for general choices of $(\alpha,\b^{(1)})$,
we may see non-zero
terms at 4th order in $t$ when looking at
the $6$th order Taylor expansion of
$g(t):=f(\x(t),\y(t))$. So this is not an indicative family.

Finally, when choosing a family,
one might end up with an indicative
family, but one which always
has must some parameter choices where
$g(t)$ vanishes at order $6$.
Such a family
will never be useful in establishing a strict local minimum.

For example, assume that the nullity of
the Hessian is greater than $1$, and
let us look at the family
\ba
\x(t;\alpha,\a^{(3)}) &:=&   \alpha^2 \a^{(2)} t^2 + \a^{(3)} t^3\\
\y(t;\alpha,\b^{(2)}) &:=&  \alpha \a^{(1)} t  + \b^{(2)} t^2
\ea
where we allow $\b^{(2)}\in \RR^m$ to be a free parameter along with $\a^{(3)}$
and $\alpha$.
(If we want, we can also restrict $\b^{(2)}$ to
be orthogonal to $\a^{(1)}$.) Again we restrict
our parameters to have unit norm.
Then by setting $\alpha=0$ and $\a^{(3)}=0$ we will have the trajectory
\ba
\x(t) &:=&   0 \\
\y(t) &:=&  \b^{(2)} t^2.
\ea
Under our assumptions, there are no terms of type $y_i y_j$ or
$y_i y_j y_k$, and thus $g(t)$ will be zero at 6th order.
Morally speaking this corresponds
in rigidity, to the idea that any $(1,1)$ flex always gives rise
$(2,3)$ flex (see
Remark~\ref{rem:irrFlex}).

\subsection{Cushing's higher order derivative test}
While we do not know of a generally applicable derivative test
for orders beyond four, an important special case is completely understood.
When the nullity of the Hessian of $f$ is one,
Cushing describes a
$2k$ derivative test
for any $k$~\cite{cushing}. His method and proof
involves the use of a
clever sequence of
changes of variables that we will not describe here.
In Section \ref{sec:higher}, we describe a $2k$
derivative test that works specifically for
energies $E$ that are stiff bar.  Our test
is related to Cushing's in that our underlying family of
test trajectories at each level is actually
the same as his, under a
change of variables.  There are,
however, some key differences:
our test is expressed without a
change of variables but
rather in the original configuration
variables;
because our application does now allow for saddles,  we
don't need to consider odd terms in the Taylor expansion;
and importantly, for our special case of a stiff-bar energy,
the proof that the test trajectories are indicative is
simpler and more direct.

\section{Rigidity certificates from energy analysis}
Now we are ready to prove our main results.
In particular, we will use the
derivative tests from the previous section 
to certify the rigidity of a framework and to compute its rigidity order.

\subsection{First-order rigidity}
We first prove Theorem \ref{thm:main0}, which is very
similar to the well-known fact (see, e.g., \cite{pss})
that, if the Hessian of a stiff-bar energy at $\p$ is positive
definite, then
the framework $(G,\p)$ is rigid.

\begin{lemma}\label{lem: stiff-bar Hessian K}
  Let $(G,\p)$ be a framework and $E$ an energy that is stiff bar at $\p$.
  Any configuration $\p'$ is in the kernel of the Hessian $H_E(\p)$
  of $E$ at $\p$, if and only if $\p(t) = \p + \p't$ is a
  $(1,1)$-flex of $(G,\p)$; i.e., $\p'\in K$.
\end{lemma}
\begin{proof}
  Let $\p(t) = \p + \p't + \cdots$ be a trajectory.  The coefficient of
  $t^2$ in the Taylor expansion of $E(\p(t))$ in $t$ at $t=0$
  is $\frac{1}{2}(\p')^T H_E(\p)\p'$, which vanishes if and only
  if $\p'$ is in the kernel of the PSD matrix $H_E(\p)$. Equivalently,
  $\p(t)$ is a $(1,2)$ $E$-flex of $(G,\p)$ if and only if $\p'$
  is in the kernel
  of $H_E(\p)$.  By Theorem \ref{thm:coin}, $\p(t)$ is a $(1,2)$
  $E$-flex of $(G,p)$
  if and only if it is a $(1,1)$-flex of $(G,\p)$.
\end{proof}
Now we can prove Theorem \ref{thm:main0} using critical
point analysis of an energy.
\begin{proof}[Proof of Theorem \ref{thm:main0}]
  To study the energy near $\p$ using indicative trajectories,
  let us define a function
  \[
    f(\delta\p) = E(\p + \delta \p) - E(\p).
  \]
  By Theorem \ref{thm: E min and rigidity}, the function
  $f$ has a local minimum at the origin with $f(0)=0$.  Moreover, the
  origin is a strict local minimum for $f$ if and only
  if $\p$ is a strict local minimum for $E$, and this
  is equivalent to $(G,\p)$ being rigid.

  The trajectories
  $\delta\p(t) = \p't$, with $|\p'| = 1$
  are indicative at order $2$ for
  $f$ at the origin.  If $(G,\p)$ is {\em not}
  rigid, the local minium of $f$ at the origin
  is not strict (Theorem~\ref{thm: E min and rigidity}),
  and so the second derivative test will be
  inconclusive (Proposition~\ref{prop:indic}).
  Hence, there is a $\p' \neq 0$ so that $f(\p' t)$ vanishes
  through second-order in $t$; i.e., $\p'$ is in the
  kernel of $H_f(0)$.
  Since $H_f(0) = H_E(\p)$,
  Lemma \ref{lem: stiff-bar Hessian K} implies
  that the trajectory $\p(t) = \p + \p't$ is
  a $(1,1)$-flex of $(G,\p)$.  This is the contrapositive
  of the first part.

  For the second part, the trajectory $\delta \p(0) = \p' t$
  is a $(1,2)$ $f$-flex, from which we conclude that
  $\p(t) = \p + \p' t$ is a $(1,2)$ $E$-flex of $(G,\p)$.
  By Lemma \ref{lem:eflex}, $E$ grows sometimes-$4$-slowly at $\p$.
  It follows that the rigidity order of $(G,\p)$, if finite,
  is at least $2$.
\end{proof}

\subsection{Second-order rigidity}
Now we prove Theorem \ref{thm:main}.  To do this
we apply our 4th derivative test, with an
approach that parallels the argument immediately above.
Let
$(G,\p)$ be a framework and let $E$ be an
energy that is stiff-bar at $\p$.  Let us first
define a function
\begin{equation}\label{eq:f2}
  f(\delta \p) = E(\p + \delta \p) - E(\p),
\end{equation}
that moves the critical point of $E$ at $\p$ to
a critical point of $f$ at the origin (to better match the setup
for the 4th derivative test).
We now define a family of trajectories 
$\{\delta\p(t;\a^{(2)},\a^{(1)})\}$ 
$\{\delta\p(t;\a^{(2)},\a^{(1)})\}$ 
by
\begin{equation}\label{eq:del2}
  \delta \p(t;\a^{(2)},\a^{(1)}) := \a^{(1)} t + \a^{(2)} t^2,
\end{equation}
where the parameters $\a^{(1)}$ and $\a^{(2)}$ satisfy:
\begin{equation}\label{eq:del2params}
  \a^{(1)}\in K, \quad \a^{(2)}\in \overline{K},
  \quad \text{and } |\a^{(1)}|^2 + |\a^{(2)}|^2 = 1.
\end{equation}
Immediately from the definitions:
\begin{lemma}\label{lem: e-flex / f-flex}
  Let $(G,\p)$ be a framework and $E$ an energy that is stiff-bar at
  $\p$.  Then a trajectory $\delta\p(t;\a^{(2)},\a^{(1)})$ in \eqref{eq:del2},\eqref{eq:del2params} is a
  $(j,k)$ $f$-flex at the origin if and only if
  $\p(t;\a^{(2)},\a^{(1)}) = \p + \delta\p(t;\a^{(2)},\a^{(1)})$ is a $(j,k)$ $E$-flex
  at $\p$.  In particular, the tight growth order of $E$ at $\p$ is
  equal to the tight
  growth order of $f$ at the origin.
\end{lemma}
Now we check that, under appropriate conditions, the trajectories
$\delta\p$ are indicative at order $4$.
\begin{lemma}\label{lem: delta p indicative}
  Let $(G,\p)$ be a framework and $E$ an energy that is stiff-bar
  at $\p$.  If $(G,\p)$ has a $(1,2)$ $E$-flex, then the family
  of trajectories $\{\delta\p(t;\a^{(2)},\a^{(1)})\}$ in \eqref{eq:del2},\eqref{eq:del2params} is
  indicative at order $4$ for the function $f$ in \eqref{eq:f2},
  at the origin.
\end{lemma}

\begin{proof}
  The idea is to apply Lemma \ref{lem:ind2a} to $f$ after a change of
  variables so that the $\y$-coordinates span $K$ and the $\x$-coordinates
  span $\overline{K}$.  The hypothesis that $(G,\p)$ has a $(1,2)$
  $E$-flex corresponds to the hypothesis of Lemma \ref{lem:ind2a}
  that says the 2nd derivative test is inconclusive at the origin.  Now
  we give the formal details.

  Let $\p(t) = \p + \a^{(1)} t +\cdots $ be a $(1,2)$ $E$-flex of
  $(G,\p)$, which exists by hypothesis.
  By Lemma \ref{lem: e-flex / f-flex},
  the trajectory $\delta \p(t) = \a^{(1)} t+\cdots$ is a $(1,2)$ $f$-flex at the
  origin.  Because $f(\delta\p(t))$ vanishes through second-order in
  $t$, the 2nd derivative test must be inconclusive for $f$ at the
  origin.  Moreover, the Hessian of $f$ at the origin is equal
  to the Hessian of $E$ at $\p$, which implies that its kernel is
  $K$.  This establishes the assumption (A1) used in
  Lemma \ref{lem:ind2a}.

  The next hypothesis of Lemma \ref{lem:ind2a}
  is that, for all trajectories in the family
  $\{\delta\p(t;\a^{(2)},\a^{(1)})\}$, in the Taylor expansion
  \[
    f(\delta\p(t;\a^{(2)},\a^{(1)})) = a_3t^3 + a_4t^4 + \cdots,
  \]
  we have $a_3 = 0$.  This follows from the fact that $f$ is non-negative
  in a neighborhood of the origin, which, in turn follows from
  Theorem \ref{thm: E min and rigidity}, since $E$ is stiff-bar at $\p$ (see Remark~\ref{rem:minNoy3}).
  (This also follows directly from Theorem \ref{thm:coin}, since
    $\p(t)$ is a (1,2) E-flex, so the theorem implies it is also a
  (1,3) E-flex.)
  Finally, in pinned configuration space, we pick an orthonormal basis
   $\{\v_1, \ldots, \v_n\}$
for $\overline{K}$ and
an orthonormal basis
 $\{\u_1, \ldots, \u_m\}$
for ${K}$.
 We now define a function
  $\tilde{f}(\x,\y)$ by
  \[
    \tilde{f}(\x,\y) =
    f(x_1\v_1 + \cdots + x_n\v_n + y_1\u_1 + \cdots + y_m\u_m).
  \]
  This function $\tilde{f}$ satisfies the assumption (A2) in Section \ref{Sec:4test}.  Because it
  is related to $f$ by a linear change of coordinates, $\tilde{f}$
  also satisfies (A1). 
  Since $f$ is not a saddle, neither is $\tilde{f}$
  and so  the third order vanishing condition on
  its Taylor expansion must be satisfied.

  We may now apply Lemma \ref{lem:ind2a} to $\tilde{f}$.  The
  indicative trajectories $\{(\x(t;\b^{(2)}),\y(t;\b^{(1)})\}$ in \eqref{eq:k2famXY}-\eqref{eq:k2famXYb} for
  $\tilde{f}$, (with $\b^{(2)} \in \RR^n$ and $\b^{(1)} \in \RR^m$)
  correspond to the trajectories
  $\delta\p(t;\a^{(2)},\a^{(1)})$ 
  (with $\a^{(2)} \in \overline{K}$ and $\a^{(1)} \in K$)
  as follows.  Let
  ${\bf V}$ be a matrix with columns $\v_i$ and let
  ${\bf U}$ be a matrix with columns $\u_j$.  We then
  have $\a^{(1)}:= {\bf U}\b^{(1)}\in K$ and $\a^{(2)}:={\bf V}\b^{(2)}\in \overline{K}$.
  Since the matrices $\bf U$, $\bf V$ are orthonormal,
  the condition $|\b^{(1)}|^2 + |\b^{(2)}|^2 = 1$ implies that
  $|\a^{(1)}|^2 + |\a^{(2)}|^2 = 1$; hence, give an indicative trajectory
  $(\x(t;\b^{(2)}),\y(t;\b^{(1)})$, we have produced one of the
  trajectories $\delta\p(t;\a^{(2)},\a^{(1)})$.
  This correspondence is a bijection.  We conclude that
  the family $\{\delta\p(t;\a^{(2)},\a^{(1)})\}$ is indicative at
  order $4$ for $f$.
\end{proof}
We now have all the results we need to prove Theorem \ref{thm:main}.
\begin{proof}[Proof of Theorem \ref{thm:main}]
  Let $(G,\p)$ be as in the statement and let $E$ be an energy that
  is stiff-bar at $\p$.  Since $(G,\p)$ has a $(1,1)$-flex, by Theorem
  \ref{thm:coin}, it has a $(1,2)$ $E$-flex.
  Thus, from Lemma~\ref{lem:eflex},
  the rigidity order of $(G,\p)$ must be at least $2$.
  Additionally, we may apply
  Lemma \ref{lem: delta p indicative} to conclude that the
  family of trajectories $\{\delta\p(t;\a^{(2)}, \a^{(1)})\}$ in \eqref{eq:del2}, \eqref{eq:del2params} is
  indicative at order  $4$  for the function $f$ at the origin.

  Suppose that the rigidity order of $(G,\p)$ is greater than $2$.
  This implies that $f$ does not grow always-$4$-quickly at the
  origin, so the
  4th derivative test must report ``inconclusive''.
  Therefore there is an indicative trajectory
  $\delta\p(t;\a^{(2)},\a^{(1)})$ such that
  $f(\delta\p(t;\a^{(2)},\a^{(1)}))$ vanishes through
  4th order in $t$; i.e., $\delta\p(t;\a^{(2)},\a^{(1)})$ is a
  $(j,4)$ $f$-flex for some $j$.  By Lemma \ref{lem:isactive},
  we must have $j=1$, showing that $\delta\p(t;\a^{(2)},\a^{(1)})$
  is a $(1,4)$ $f$-flex at the origin.

  The definition of $f$ then implies that
  $\p(t) = \p + \delta\p(t;\a^{(2)}, \a^{(1)})$
  is a $(1,4)$ $E$-flex of $(G,p)$.
  Finally, by Theorem \ref{thm:coin}, $\p(t)$ is a $(1,2)$-flex of $(G,\p)$.
  This is the contrapositive of the first part.

  For the second part, if $(G,\p)$ has a $(1,2)$-flex, then Theorem
  \ref{thm:coin} implies that it has a $(1,4)$ $E$-flex.  Lemma
  \ref{lem:eflex} then implies that $E$ grows sometimes-$6$-slowly
  at $\p$.  Hence, if $(G,\p)$ is rigid, its rigidity order is at least
  $3$.
\end{proof}

\subsection{Higher order rigidity when $\dim(K)=1$}
\label{sec:higher}
Now we give a new proof of Theorem~\ref{thm:alex} by
analyzing the critical points of a stiff bar energy.
We recall that the key hypothesis of Theorem
\ref{thm:alex} is that, for the framework $(G,\p)$
under consideration, we have $\dim(K) = 1$.
For the rest of this section, we will assume that
$\dim(K)=1$.  We now describe a sequence of $2k$
derivative tests that apply to this case.

Let $E$ be a stiff bar energy.
Suppose that a framework $(G,\p)$ has a $(1,2)$
flex $\p(t) := \p + \a^{(1)}t + \a^{(2)}t^2$,
which we can choose such that
$|\a^{(1)}|=1$  (using a reparameterization in $t$),
and
$\a^{(2)} \in \overline{K}$
(using Lemma~\ref{lem:kcomp}). For this setup, $K$ is the span of $\a^{(1)}$.
Let
$y \in \RR$  and 
$\a^{(3)} \in \overline{K}$
be parameters.

Again, we define
\ba
f(\delta \p) := E(\p + \delta \p)-E(\p).
\ea
Let us define the family
of test trajectories
\bna
\label{eq:k3fam}
\delta \p(t;\alpha,\a^{(3)})
:= \alpha \a^{(1)}t + \alpha^2 \a^{(2)}t^2 +\a^{(3)} t^3
\ena
in which $\a^{(1)}$ and $\a^{(2)}$ are fixed, and the parameters $\alpha$
and $\a^{(3)}$ are on  the ``unit sphere''  i.e.:
\ba
\alpha^2 +
|\a^{(3)}|^2  =1.
\ea

\begin{lemma}\label{lem42}
  Assume that $\dim(K)=1$ and that
  $\p(t) := \p + \a^{(1)} t + \a^{(2)} t^2$ is a $(1,2)$ flex.
  Then the family
  described by Equation (\ref{eq:k3fam}), parameterized
  by $(\alpha,\a^{(3)})$, is indicative
  at order $6$ for $f$.
\end{lemma}

\begin{proof}
  The unit sphere is compact (property (I1)),
  and the parameterization is
  plainly smooth (property (I2)).
  Now we establish
  the 5-vanishing property (I4). There are
  two cases, depending on whether $\alpha = 0$.
  When
  $\alpha\neq 0$,
  then when  $\p(t)$ is a $(1,2)$ flex,
  so too
  is $\p(\alpha t)$ (by reparameterization), and so too is
  $\p(\alpha t) + \a^{(3)} t^3$
  (by truncation
  (see Remark~\ref{rem:trunc})).
  So from Theorem~\ref{thm:coin},
  $\p(\alpha t) + \a^{(3)} t^3$ is also a $(1,5)$ E-flex.
  When $\alpha=0$, then $\p+\a^{(3)}t^3$ is technically a $(3,2)$  flex,
  and thus also a $(3,5)$ E-flex.
  Thus each trajectory
  $\delta \p(t;\alpha,\a^{(3)})$
  is a $(j,5)$ f-flex for some $j$.


  The most involved step is to establish
  the initial coverage property (I3).
  To this end,  
  let us write a point in pinned configuration space
  as $\x+\y$ with  $\y \in K$ and $\x \in \overline{K}$.
  For a fixed $s \neq 0$, define the surface $H_s$ 
  in pinned configuration space,
  to be the set of points 
  $\x+\y$
  that satisfy
  \ba
  {s^4} |\y|^2 +
  |\x - ( (\y\cdot \a^{(1)})^2 \a^{(2)} )|^2 - s^6 = 0.
  \ea
  The interior of $H_s$ consists of  points
  where the left hand side is negative,
  while the exterior of $H_s$ consists of  points
  where the left hand side is positive.

  Still keeping $s$ fixed, define the invertible map
  taking $(\alpha,\a^{(3)})$ to $\x+\y$
  with
  \begin{align*}
      \x &:=\alpha^2\a^{(2)}s^2 +\a^{(3)}s^3\\
      \y &:= \alpha  \a^{(1)} s
  \end{align*}
  Then $\x,\y$ so defined are on the surface $H_s$, if and only if $(\alpha,\a\me{3})$ are on the unit sphere in parameter space, as can be verified using the following algebra:
  \ba
  s^4 |\y|^2 +
  |\x - ((\y\cdot \a^{(1)})^2 \a^{(2)}) |^2 - s^6
  &=&
  s^4 |\alpha\a^{(1)}s|^2 +
  |\alpha^2\a^{(2)}s^2
  +\a^{(3)}s^3 - ((\alpha\a^{(1)}s \cdot \a^{(1)})^2 \a^{(2)})|^2 - s^6 \\
  &=&
  s^6 \alpha^2 +
  |\alpha^2\a^{(2)}s^2
  +\a^{(3)}s^3 - (\alpha^2  s^2 \a^{(2)} ) |^2 - s^6 \\
  &=&
  s^6 \alpha^2 +
  |\a^{(3)}s^3|^2  - s^6 \\
  &=&
  s^6 (\alpha^2 + |\a^{(3)}|^2  - 1).
  \ea
 
  Let $0<\delta <\eps $ be given.
  Let us consider the surface  $H_\delta$.
  For a small enough neighborhood $U$ of the origin,
  each point $\x+\y \in U$ is in the interior of $H_\delta$.
  Meanwhile, for each point $\x+\y \in U$ and not equal to
  $0$, there must be a sufficiently small value
  $r>0$ so that $\x+\y$ is in the exterior of $H_r$.
  Thus by continuity, there must be an intermediate
  value $r<t<\delta$ such that $\x+\y$ is on $H_t$.
  Thus for this $t$, we have
  $\x =\alpha^2\a^{(2)}t^2 +\a^{(3)}t^3$ and
  $\y = \alpha  \a^{(1)} t$,
  for some
  $(\alpha,\a^{(3)})$ on the unit sphere.
\end{proof}

The same argument applied to general $k$ gives us:
\begin{lemma}
  \label{lem:standInd}
  Assume $\dim(K)=1$.
  Suppose that $(G,\p)$ has a
  $(1,k-1)$ flex
  \ba
  \p(t) := \p + \a^{(1)}t +  \a^{(2)}t^2 +...+\a^{(k-1)} t^{k-1}
  \ea
  with $|\a^{(1)}|=1$ and
  with
  $\a^{(j)} \in \overline{K}$ for
  $j\ge 2$.
  Then the family
  \bna
  \label{eq:familyk}
  \delta\p(t;\alpha,\a^{(k)}) := \alpha \a^{(1)}t + \alpha^2 \a^{(2)}t^2 +...+ \alpha^{k-1}
  \a^{(k-1)}t^{k-1} + \a^{(k)} t^{k}
  \ena
  parameterized over $(\alpha,\a^{(k)})$ 
  in
  $\RR \times \overline{K}$
  with $|\alpha|^2+|\a^{(k)}|^2=1$
  is indicative at order $2k$
  for $f$.
\end{lemma}
\begin{proof}
  The proof follows the same logic as that of Lemma~\ref{lem42}.
  For general $k$ we define the surface $H_s$ using
  \ba 
  {s^{2k-2}} |\y|^2 +
  |\x - (    (\y\cdot \a^{(1)})^2 \a^{(2)}     +\cdots +     (\y\cdot \a^{(1)})^{(k-1)} \a^{(k-1)}   )|^2 - s^{2k} = 0.
  \ea
  The map
  taking $(\alpha,\a^{(k)})$ to $\x+\y$
  is defined as
  \ba
  \x :=
  \alpha^2 \a^{(2)}s^2 +...+ \alpha^{k-1}
  \a^{(k-1)}s^{k-1} + \a^{(k)} s^{k}, \qquad
  \y := \alpha  \a^{(1)} s.
  \ea
\end{proof}

\begin{lemma}
  \label{lem:isactiveK}
  Let $(G,\p)$ be a framework with $\dim(K)=1$ that has a $(1,k-1)$
  flex.  Suppose further that the
  $2k$th derivative test applied to the function $f$ using the
  family of indicative trajectories
  defined in \eqref{eq:familyk} is inconclusive.  Then, if
  \[
    \delta\p(t;\alpha,\a^{(k)}) := \alpha \a^{(1)}t + \alpha^2 \a^{(2)}t^2 +...+
    \alpha^{k-1}\a^{(k-1)}t^{k-1} + \a^{(k)} t^{k}
  \]
  is an indicative trajectory such that
  $f(\delta\p(t;\alpha,\a^{(k)}))$ vanishes through order $2k$, then
  $\alpha\neq 0$.  This, in particular, implies that the associated trajectory
  \[
    \p(t) := \p + \delta\p(t;\alpha,\a^{(k)})
  \]
  is $1$-active.
\end{lemma}
\begin{proof}
  Let $\delta\p(t;\alpha,\a^{(k)})$ be as in the statement.
  Suppose, for a contradiction, that $\alpha = 0$. Let $\tilde{f} :=
  f|_{\overline{K}}$.
  Since the Hessian of $f$ is definite on $\overline{K}$, the second
  derivative test will certify that $\tilde{f}$ has a strict local
  minimum at the origin.  Now define a trajectory
  $\delta\tilde{\p}(\tau(t);0,\a^{(k)}) := \tau(t)\a^{(k)}$
  with $\tau(t) := t^{k}$.  The trajectory
  $\delta\tilde{\p}(\tau(t))$ is linear in $\tau$, and
  so indicative for the second derivative test.  By construction
  $f(\delta\tilde{\p}(\tau(t)))$
  vanishes through second order in $\tau$, making the second
  derivative test inconclusive for
  $\tilde{f}$.  The resulting contradiction implies that the
  assumption $\alpha = 0$ was false.
\end{proof}

At this point we have the setup to prove
Theorem~\ref{thm:main2}.  The argument is structurally
similar to the proof of Theorems \ref{thm:main0} and \ref{thm:main}.

\begin{proof}[Proof of Theorem~\ref{thm:main2}]
  Let $(G,\p)$ be as in the statement and let $E$ be an energy that
  is stiff-bar at $\p$.  
      The framework $(G,\p)$ has a $(1,k-1)$-flex
  which we can choose such that
$|\a^{(1)}|=1$  (using a reparameterization in $t$),
and
$\a^{(j)} \in \overline{K}$
for $j>1$
(using Lemma~\ref{lem:kcomp}).

  Since $(G,\p)$ has a $(1,k-1)$-flex, by Theorem
  \ref{thm:coin}, it has a $(1,2k-2)$ $E$-flex.
  Thus, from Lemma~\ref{lem:eflex},
  the rigidity order of $(G,\p)$ must be at least $k$.
  Additionally, we may apply Lemma \ref{lem:standInd} to
  conclude that the family of trajectories
  $\{\delta\p(t;\alpha,\a^{(k)})\}$ is indicative
  at order $2k$ for the function $f$ at the origin.

  Suppose that the rigidity order of $(G,\p)$ is greater than $k$.
  This implies that $f$ does not grow always-$2k$-quickly at the
  origin. Because the trajectories $\{\delta\p(t;\alpha,\a^{(k)})\}$
  are indicative at order $2k$ for $f$ at the origin and $f$ is
  non-negative, Proposition \ref{prop:indic} implies that, for some
  parameters $\alpha$ and $\a^{(k)}$,
  $f(\delta\p(t;\alpha, \a^{(k)}))$ vanishes through order $2k$.
  Hence, $\delta\p(t;\alpha, \a^{(k)})$ is a $(j,2k)$ $f$-flex for
  some $j$.  By Lemma \ref{lem:isactiveK}, we have $\alpha\neq 0$,
  so $\delta\p(t;\alpha, \a^{(k)})$ is a $(1,2k)$ $f$-flex at the
  origin.

  The definition of $f$ then implies that
  $\p(t) = \p + \delta\p(t;\alpha, \a^{(k)})$ is a $(1,2k)$ $E$-flex of
  $(G,\p)$.  Finally, by Theorem \ref{thm:coin}, $\p(t)$ is a
  $(1,k)$-flex of $(G,\p)$.
  This is the contrapositive of the first part.

  For the second part, if $(G,\p)$ has a $(1,k)$-flex, then Theorem
  \ref{thm:coin} implies that it has a $(1,2k)$ $E$-flex.  Lemma
  \ref{lem:eflex} then implies that $E$ grows sometimes-$2k+2$-slowly
  at $\p$.  Hence, if $(G,\p)$ is rigid, its rigidity order is at least
  $k+1$.
\end{proof}

\begin{remark}
  \label{rem:genCush}
  Even assuming that $\dim(K)=1$,
  the vanishing property, required for
  indicativeness of the family
  of Equation (\ref{eq:familyk}),
  is harder to directly prove for a general
  energy function $f(\q)$.
  In particular,  one has to rule out the possibility that
  due to the inclusion of  the
  $\a^{(k)}t^k$ in $\delta(\p)(t)$, that
  $f(\delta \p(t))$  no longer vanishes through
  order $2k-2$.
  It turns out that the vanishing property of this family,
  for a general function $f$, in the case
  $\dim(K)=1$,
  can actually be proven using a clever series of
  variable changes described  by Cushing~\cite{cushing}.
  It also turns out that
  Cushing's sequence of  variable changes
  can be
  interpreted as a sequence of ``singularity blow-ups''
  applied to the
  variety $f=0$~\cite{greg}.
\end{remark}

\section{Examples}

\begin{figure}[ht]
  \centering
  \includegraphics[scale=0.5]{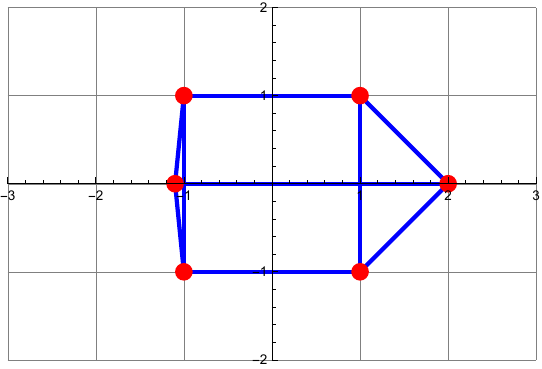}
  \includegraphics[scale=0.5]{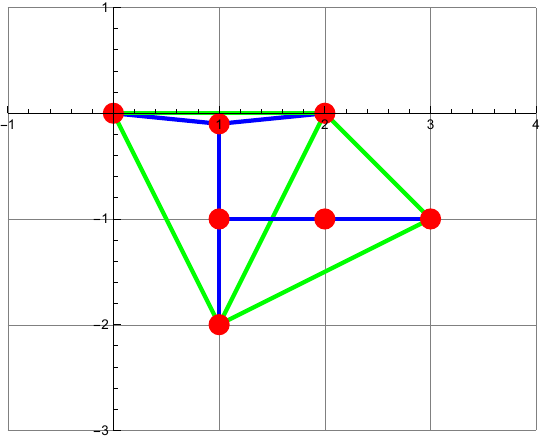}
  \captionsetup{labelsep=colon,margin=1.3cm}
  \caption{
    Left:  The half flat prism has rigidity order $4$.
    The middle left vertex has been shifted for visualization
    purposes.
    Right: The Leonardo-3 framework~\cite{leonardo} has a rigidity order
    of $8$. The upper middle vertex has been shifted for visualization
    purposes. The green edges are used so that no
    pins are needed in the framework.
    The rigidity orders of these examples persist under
    projective transformations.
  }
  \label{fig:k33}
\end{figure}

\begin{figure}[ht]
  \centering
  \includegraphics[scale=0.5]{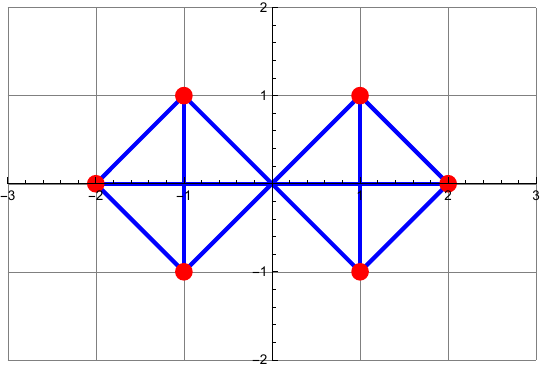}
  \includegraphics[scale=0.5]{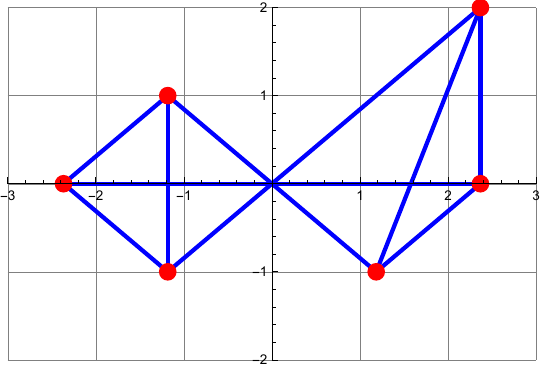}
  \includegraphics[scale=0.5]{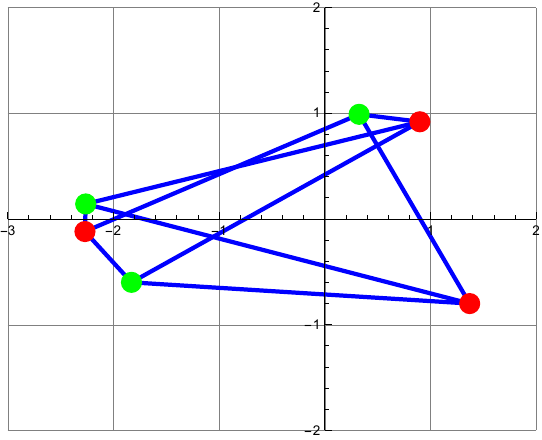}
  \captionsetup{labelsep=colon,margin=1.3cm}
  \caption{Left: This symmetric
    flipped triangular prism has a rigidity order of $4$.
    Middle: This non-symmetric flipped prism has a rigidity
    order of $3$. Right: This framework of $K_{3,3}$ has a
    rigidity order of $3$. (The bipartite partitions are visualized
    in red and green.)
    All three become second-order rigid
    under a generic affine transform.
  }
  \label{fig:xnan2}
\end{figure}

\begin{figure}[ht]
  \centering
  \includegraphics[scale=0.5]{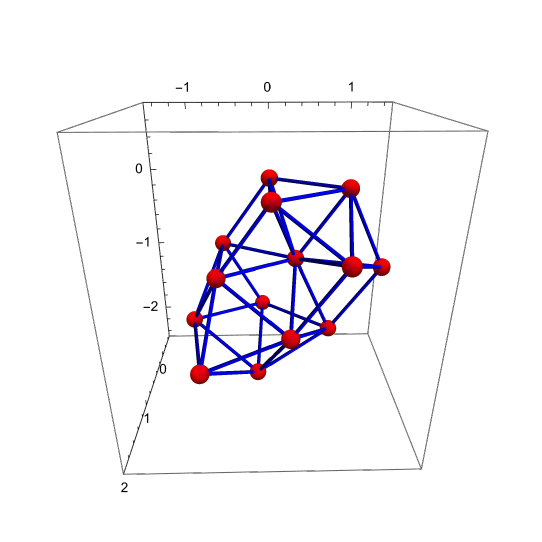}
  \includegraphics[scale=0.5]{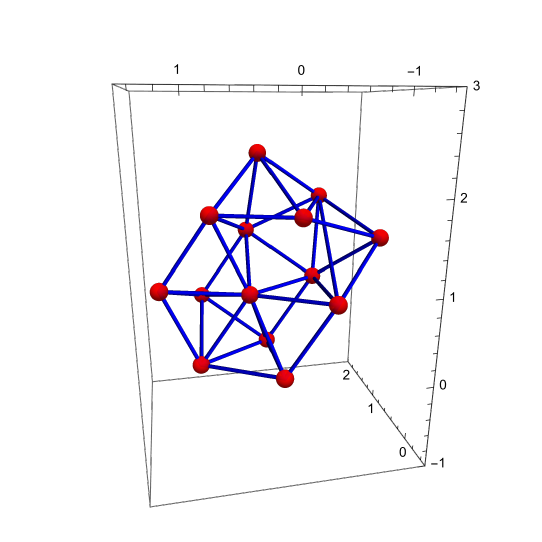}
  \includegraphics[scale=0.5]{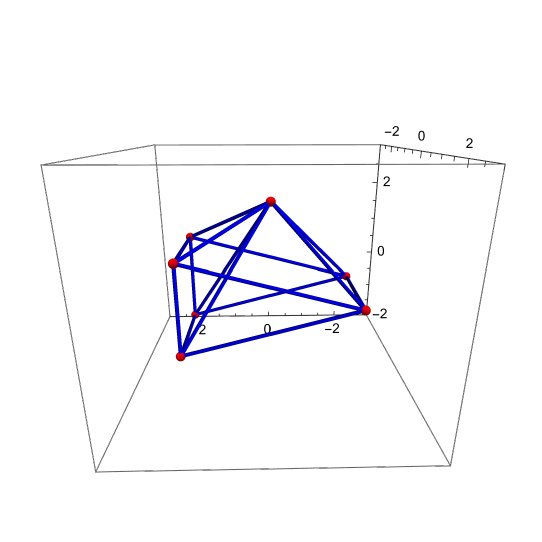}
  \captionsetup{labelsep=colon,margin=1.3cm}
  \caption{
    Left, Middle: These frameworks arise
    from packings of
    $14$ identical spheres.  They were found in
    \cite{holmes2016enumerating} but were discarded from the
    data presented in that paper because they did not pass the
    test for prestress stability.
    Both have a rigidity order of $3$.
    Right: A framework of the coned triangular prism.
    It has a rigidity order of $4$.
    All three become second-order rigid
    under a generic affine transform.
  }
  \label{fig:sph2}
\end{figure}

In Figures~\ref{fig:k33}, \ref{fig:xnan2} and \ref{fig:sph2} we
show a number of two and three
dimensional frameworks and discuss their rigidity orders. All
have $\dim{K}=1$, so their rigidity orders
can efficiently be (numerically) computed by iteratively
solving \eqref{eq:flex} for $(1,k)$ flexes until one is not
found. The coordinates for these examples are given in Appendix
section \ref{sec:coords}.
The examples in Figures~\ref{fig:xnan2}(mr) and~\ref{fig:sph2}(r) were produced using  
an equilibrium stress based technique. As stresses are beyond the scope of this 
paper, we defer the explanation of that technique to our companion paper~\cite{HOpre}.

The rigidity order of the frameworks
in Figures~\ref{fig:k33} and ~\ref{fig:sph2} drop down to $2$ under
generic affine transformations. The
exact coordinates used here were found
using a numerical optimization, and
are available in Appendix~\ref{sec:coords}.

\section{Extensions to more general constraints}

So far our exposition has considered vertices with distance
constraints between pairs of vertices.
However, constraints arise in a number of other systems for
which one might be similarly interested in the system's
rigidity or the behavior of an energy function near a critical
point. For example, one may wish to consider points constrained
to lie on a particular edge \cite{santangelo1}, or rigid bodies
in contact \cite{donev2007underconstrained}, or area and
perimeter constraints on collections of vertices,  such as
arise in  ``vertex models'' used to study tissue mechanics
\cite{manning2024rigidity}. Our results extend without any
difficulty to a more general setup; in this section we outline how.

In the general setup, we have a set of \defn{configuration
variables} $\q\in \R^n$, and a set of \defn{measurement
variables} $\m(\q)=(m_\alpha(\q))_{\alpha\in \mathcal A}$,
where the functions $m_\alpha:\R^n\to \R$ are analytic at some
configuration $\p$, and $\mathcal A$ is a finite set of indices.
An analytic trajectory $\p(t)$ with $\p(0)=\p$ is a
\defn{$(j,k)$-flex} (at $\p$) if $\p(t)$ is $j$-active and
$\m(\p(t))$ is $k$-vanishing.

We are interested in the rigidity of the pair $(\m,\p)$ of
measurement variables with a specific configuration. We say
$(\m,\p)$  is  \defn{rigid}  if there does not exist a
continuous trajectory $\p(t)$  at $\p$ such that $\m(\p(t)) =
\m(\p)$ for all $t$ in a neighborhood of 0, except when $\p(t)=
\p$ for all $t$.  We do not concern ourselves here with
generalizations of pinning, as we assume
for simplicity that our measurements are not
invariant with respect to isometries
acting on $\q$ (for example, our problem includes
sufficiently many boundary constraints.)

We furthermore have a \defn{measurement-based energy}
\[
  E(\q) = \sum_{\alpha \in \mathcal A} E_{\alpha}(m_\alpha(\q)),
\]
and we say $E$ is \defn{stiff-bar for $(\m,\p)$} if each
$E_\alpha$ is analytic at $\p$, and has a strict local minimum
with positive curvature at $m_\alpha(\p)$. As before, a
\defn{$(j,k)$ E-flex} is a $j$-active trajectory $\p(t)$ with
$\p(0)=\p$ such that $E(\p(t))$ is $k$-vanishing.

An energy $E$ which is stiff-bar at $\p$ is analytic at $\p$,
so all of our results apply, connecting flexes, E-flexes,
rigidity orders, and energy growth, as these results did not
depend on the specific forms of $E$ or $\m$, but only on their
analyticity. In particular, we have the following results:
\begin{itemize}
  \item (Theorem \ref{thm:coin}.) A trajectory $\p(t)$ is a
    $(j,k)$ flex iff it is a  $(j,2k)$ E-flex.
  \item (Theorem \ref{thm: order jk-flex}.) If $(\m,\p)$ is
    rigid, then there exists an $s\in\mathbb Q$ such that every
    stiff bar energy for $(\m,\p)$ has tight growth order of
    $s$. Furthermore, $s$ may be computed by considering the
    $(j,k)$ flexes, as in \eqref{sflex}.
    Therefore, the \defn{rigidity order} $\nu=s/2$ of $(\m,\p)$
    (Definition \ref{defn:order}) is well-defined.
  \item (Theorems \ref{thm:main0}, \ref{thm:main},
    \ref{thm:main2}.) These  apply  verbatim, and connect the
    existence or not of certain $(j,k)$ flexes, to the rigidity
    order of $(\m,\p)$.

\end{itemize}

The only part of our theory which changes, with a change in the
measurement variables, is the equations used to solve for
flexes. Instead of \eqref{eq:flex}, one will have a different
system of equations, obtained by taking time derivatives of the
functions $\{m_\alpha(\p(t))\}_{\alpha \in \mathcal A}$.

\def\v{\oldv}

\appendix

\section{Rigidity and pinning}
\label{sec:pin}
In order to understand pinning, we first describe a procedure 
to convert any configuration $\p$ into a 
pinned configuration $\q$ by applying a translation, followed by 
a sequence of rotations around (carefully chosen) coordinate  
subspace axes.  If $T$ is an isometry of $\RR^d$, we will denote
by $T\p$ the application of $T$ to each of the $\p_i$.

\begin{lemma}\label{lem: pin by isom}
Let $\p$ be a  configuration in $\RR^d$. 
Then there is an isometry $T$ of $\RR^d$ such that $T\p$ is pinned.
\end{lemma}
\begin{proof}
To set up some notation, we denote by $\{e_1, \ldots, e_d\}$
the standard basis vectors, and for a set $E$ of the basis 
vectors, by $\R^E$ the corresponding coordinate subspace.
By $R_{i,j}(x,y)$, where we assume $y \neq 0$,
we denote the rotation around the 
axis $(\R^{\{e_i,e_j\}})^\perp$ that sends the 
vector $xe_i + ye_j$ to $(x^2 + y^2)^{1/2}e_i$.
For example, in $\R^4$, 
\[
    R_{3,4}(x,y) = 
    \begin{bmatrix}
        1 & 0 & 0 & 0 \\
        0 & 1 & 0 & 0 \\
        0 & 0 & \frac{x}{(x^2 + y^2)^{1/2}} & \frac{y}{(x^2 + y^2)^{1/2}} \\
        0 & 0 & \frac{-y}{(x^2 + y^2)^{1/2}} & \frac{x}{(x^2 + y^2)^{1/2}} \\
    \end{bmatrix}.
\]

The first step in computing an isometry $T$ is to translate $\p$ so that  $\p_1$ moves
to the origin.  
The remaining steps are defined inductively:
Suppose that 
for some $1\le i < D$,
$\p$ is a 
configuration 
such that, for $1\le j \le i$, $\p_j$ is in the span of the first $j-1$ 
coordinate vectors (when $j-1=0$ we take this span to be just  $\{{\bf 0}\}$).  
Now we will apply a 
series of rotations to $\p$ so that $\p_{i+1}$ lies in 
the span of the first $i$ coordinate vectors; i.e., these rotations zero out the 
last $d-i$ coordinates of $\p_{i+1}$.  We will choose these rotations so that
 $\p_1, \ldots, \p_i$ are unchanged.

We find the rotations to zero out theses coordinates by an inner induction.
Let $\p_{i+1} = (x_1, \ldots, x_d)$ and suppose that, for some $i < j \le d$, 
we have $x_{j} \neq 0$, and that
all $x_k$ with $j<k \le d $ are equal to $0$.
The rotation $R_{i,j}(x_{i},x_j)$ zeros out the 
$j$th coordinate of $\p_{i+1}$. Meanwhile, it only affects the $i$th and$j$th 
coordinates in $\p$.
In particular 
it keeps
all $x_k$ with $j<k \le d $  equal to $0$
and  $\p_1, \ldots, \p_i$ remain in 
pinned position.  This closes the inner induction, and so also the outer 
one.
(The rotation $R_{m,j}(x_{m},x_{j})$
arising from 
any index, $i\le m < j$ would  work in this step, but our choice
of $i$ is useful for the proof of 
Lemma~\ref{lem:close} below).

Composing the rotations from each step gives the desired isometry $T$.
\end{proof}

A detailed analysis of the above process will let us establish the following.
\begin{lemma}
\label{lem:close}
Let $\p$ be a lead-spanning and pinned configuration.
Let $U$ be a neighborhood of $\p$. 
Then there is smaller neighborhood
$V\subseteq U$ of $\p$ with the property that any $\q\in V$
is congruent to a pinned configuration in $U$.
\end{lemma}
\begin{proof}
The goal is to show that, 
as $\q \to \p$, the transformation $T$ in the conclusion 
of Lemma \ref{lem: pin by isom} satisfies $\|T\q - \q\|\to 0$, 
from which we conclude that $\|T\q - \p\|\to 0$.

To this end, we analyze the construction from the proof 
of Lemma \ref{lem: pin by isom}, applied to 
$\q$, step-by-step.  For the first step, we observe that the 
translation $\x\mapsto \x - \q_1$ converges to the 
identity as $\q\to \p$, since $\p_1 = 0$.  Now 
we suppose that we are working on $\q_{i+1} = (x_1, \ldots, x_d)$ 
and are setting 
the variable $x_{j}$ to zero (so $j > i$ and for $j < k \le d$,
$x_k =0$).  

Assume first that  $e_i$ is in the 
affine span of $\p$.  In this case, since $\p$ is lead spanning and pinned, 
the $i$th coordinate of $\p_{i+1}$ is non-zero and its $j$th
coordinate is zero.  Hence, as $\q\to \p$, the rotation
$R_{i,j}(x_i,x_j)$
converges to the identity.

Now assume that   $e_i$ 
is outside the affine span of $\p$.  
Because $\p$ is pinned and lead spanning, its affine span 
is the coordinate subspace spanned by $e_1, \ldots, e_\ell$, for some $\ell$,
so we have $\ell < i < j$, and so $e_{j}$ is 
also outside the affine span of $\p$.
Hence, as $\q\to \p$, any point $\q_k$
approaches the null space of $R_{i,j}(x_i,x_{j}) - I$.
Since both $R_{i,j}(x_i,x_{j}) - I$ and $\q_k$ have
bounded norm, we conclude that $\|R_{i,j}(x_i,x_{j})\q_k - \q_k\|\to 0$.

These cases are exhaustive, and so the result follows.
\end{proof}
\begin{lemma}\label{lem: pinning}
  Let $\p$ be a  
  lead-spanning and
  pinned configuration.  Then
  there is a neighborhood of $\p$ that does not contain any other
  pinned configuration of $n$ points that is congruent to $\p$.
\end{lemma}
\begin{proof}
 Suppose that $\p$ has an $\ell$-dimensional affine span.
  Let
  $\q$ be a pinned configuration congruent to $\p$,
  and let $T$ be  an isometry so that $\q = T\p$.
  If $\p_i = \q_i$ for all
  $1\le i\le \ell+1$, then $T$ is the identity on the affine
  span of $\p$, and hence on $\p$.
  Thus $\p = \q$.
  To finish the proof, we just need to find a neighborhood $U$ of $\p$ such that,
  if $\q\in U$ is a pinned configuration that is
  congruent to $\p$, then $\q_i = \p_i$ for all
  $1\le i\le \ell+1$.  This we do by induction on $i$.

  Let $U_1$ be any neighborhood of $\p$.  Let $\q\in U_1$
  be  pinned and congruent to $\p$.
  Since
  $\p$ and $\q$ are  pinned, certainly $\p_1
  = \q_1$ (they
  are both the origin).  This establishes the base of the induction.

  Suppose now, for some
  $1 < i < \ell+1$, we have a neighborhood $U_i$ of
  $\p$, so that, if $\q\in U_i$ is pinned and
  congruent to $\p$, we have
  $\q_j = \p_j$ for $1 \le j \le i$.  We will show that 
  there is a neighborhood
  $U_{i+1}\subseteq U_i$ of $\p$ such that, if $\q\in U_{i+1}$
  is pinned  and congruent to $\p$, then
  $\q_{i+1} = \p_{i+1}$.  To this end, let $\x$ be a
  point in the linear span of $e_1, \ldots, e_i$ such that,
  for each $1\le j\le i$, $|\p_j - \x| = |\p_j - \p_{i+1}|$.
  Because $\p$ is lead-spanning and $i \le \ell+1$, the $\p_j$
  are affinely independent.  It now follows that either $\x = \p_{i+1}$
  or $\x$ is the reflection (in the linear space spanned by the first $i$
  coordinate vectors) through the affine span of the $\p_j$
  (which is simply the linear space spanned by the first $i-1$
  coordinate vectors). Since there is a positive distance
  between $\p_{i+1}$
  and its reflection, shrinking $U_{i}$ sufficiently to remove this
  ambiguity closes the induction.

  We stop at $U_{\ell+1}$, which proves the
  lemma.
\end{proof}

\subsection*{Proof of Proposition \ref{prop: pinned rigid}}
Let $(G,\p)$ be a lead spanning framework and $(G,\q)$ be 
a pinned framework congruent to $\p$.  One exists by Lemma 
\ref{lem: pin by isom}. Because $\p$ is lead spanning, so is $\q$.

Suppose first that $(G,\p)$ is rigid.  As $(G,\q)$ is 
congruent to $\p$, it is also rigid.  Let $U\ni \q$ be a 
neighborhood of $\q$ so that if $\r\in U$ and $(G,\r)$ is 
equivalent to $(G,\q)$, then $\r$ is congruent to 
$\q$.  Let $V\ni \q$ be a neighborhood of $\q$ such that, 
if $\r\in V$ is a pinned configuration congruent to $\q$, then 
$\r = \q$.  One exists by Lemma \ref{lem: pinning}.
Let $\r \in U\cap V$ be a pinned configuration 
such that $(G,\r)$ is equivalent to $(G,\q)$.  Since $\r\in U$, 
$\r$ is congruent to $\q$, which, because $r\in V$, impies that 
$\r = \q$.  Hence $(G,\q)$ is pinned rigid.

For the other direction, suppose that $(G,\q)$ is pinned 
rigid.  Let $U\ni \q$ be a neighborhood of 
$\q$ such that, if $\r\in U$ is pinned and equivalent to $\q$, 
then $\r = \q$; $U$ exists because $(G,\q)$ is pinned 
rigid.  Let $V$ be a neighborhood $U\supseteq V\ni \q$ be a 
neighborhood of $\q$ such that, if $\r\in V$, then 
$\r$ is congruent to a pinned configuration $\x\in U$; $V$ 
exists by Lemma \ref{lem:close}.  Let $\r\in V$ be a 
configuration such that $(G,\r)$ is equivalent to $(G,\q)$.
Because $\r\in V$, $\r$ is congruent to a pinned configuration $\x\in U$.
Since $(G,\r)$ is equivalent to $(G,\q)$, so is $(G,\x)$.  Since 
$x\in U$, $\x = \q$, and, hence, $\r$ is congruent to $\q$.  This 
shows that $(G,\q)$ is rigid. Because $(G,\p)$ is congruent to 
$(G,\q)$, $(G,\p)$ is rigid.
\hfill $\qed$.

\section{Pinning and $(j,k)$ flexes}
We now extend the results in the previous section to 
show that whether or not a framework has a $(j,k)$-flex 
does not depend on how it is pinned.
\begin{definition}
  An unpinned  trajectory $\p(t)$
  is \defn{$k$-trivial}  if it agrees through $k$th order with
  $T(t)\p(0)$
  where $T(t)$ is an analytic
  trajectory of isometries
   with $T(0)=I$.
\end{definition}
\begin{definition}
  Let $(G,\p)$ be an unpinned framework.
  Let $j,k$ be integers $\ge 1$.
  An unpinned
  trajectory $\p(t)$ is an \defn{unpinned $(j,k)$-flex} if
  $\p(t)$ is $(j-1)$-trivial, but not $j$-trivial,
  and $\m(\p(t))$ is $k$-vanishing.
\end{definition}
\begin{remark}
  This is fundamentally the same as the definition of
  Stachel~\cite{stachel,narwal}, since from
  {\cite[Lemma 4.2.1]{pss}}, the leading $k$ derivatives
  of a $k$-trivial trajectory
  can be always zeroed out under an appropriate analytic
  trajectory of isometries.
\end{remark}
With the definitions in place, we state the result of 
this section.
\begin{proposition}\label{prop:jkpin}
  Fix a dimension $d$.
  Let $(G,\p)$ be a framework such that $\p$
  is lead-spanning.
  Let $\q$ be a pinned configuration that is
  congruent to $\p$.
  Then $(G,\q)$ has an pinned $(j,k)$ flex iff
  $(G,\p)$ has an unpinned $(j,k)$ flex.
\end{proposition}

The main technical issue of the proof of this 
proposition is
dealing with 
the $j$-active property of a $(j,k)$ flex.

\begin{lemma}
  \label{lem:isoTraj}
  Let $\p(t)$ be an unpinned trajectory, with 
  $\p(0)$ lead spanning.
  Then there is an analytic trajectory of isometries
  $T(t)$,
  for $t\in[0,\eps)$ for some $\eps$,
  so that $\q(t) := T(t)  \p(t)$
  is a pinned trajectory.
\end{lemma}

The main idea of the proof of this lemma is
to simply apply the construction used
in the proof of Lemma~\ref{lem: pin by isom},
except applied to $\p_i(t)$
instead of $\p_i$.  
The only complication arises when the
affine span of $\p(0)$ is $\le d-2$.
In this case, the construction of the matrix
$R_{i,j}(x_i(t),x_{j}(t))$
can blow up at $t=0$ (with
$x_i(t)=0$ and $x_{j}(0)=0$).
We will deal with this using the following
factoring trick.

\begin{lemma}\label{lem: order}
Let $f_1, f_2 : \RR\to \RR$ be analytic at $0$,
and suppose there is a $k$ such that 
both $f_i$ vanish up to order $k-1$ at $t=0$
and that there is a $j\in \{1,2\}$ such that 
$f_j$ is  non-vanishing at order
$k$ at $t=0$.  
Write $F_i(t) := f_i(t)/t^{k}$.  
Then each $F_i(t)$ is analytic at $t=0$, and, 
additionally 
$(F_1^2(t) +  F_2^2(t))^{-1/2}$ 
is well defined and analytic 
$t=0$.
\end{lemma}
\begin{proof}
Because each of  the $f_i$ 
vanish up to order $k-1$, 
its power series has can be factored by
$t^{k}$.
Then $F_i(t):=f_i(t)/t^{k}$ 
has a power series with the same radius
of convergence as $f_i(t)$, so
$F_i(t)$
is analytic at $0$.
Since  $f_j(t)$ is non-vanishing
at order $k$ at $t=0$, we will have
$F_j(t)$ non-vanishing at $t=0$.  Hence, 
$(F_1^2(t) + F_2^2(t))$ is
non-vanishing and analytic at $t=0$.
Then so too is  
$(F_1^2(t) + F_2^2(t))^{1/2}$.
So $(F_1^2(t) + F_2^2(t))^{-1/2}$ is well defined and analytic at $t=0$.
\end{proof}
\begin{proof}[Proof of Lemma \ref{lem:isoTraj}]
We apply the construction used in the proof of Lemma
\ref{lem: pin by isom}, except applied to $\p_i(t)$
instead of $\p_i$.  
We need to to make the construction of the 
matrix $R_{i,j}(x_i(t),x_{j}(t))$  well 
defined and analytic at $t=0$.

If both $x_i(t)$ and $x_{j}(t)$ are
identically $0$, we simply skip this step. 

Otherwise, there is a 
smallest $k\ge 0$ so that at least one of 
$x_i(t)$ and $x_{j}(t)$ is non-vanishing at order $k$  at $t=0$.  
If $k=0$, we can use the rotation $R_{i,j}(x_i(t),x_{j}(t))$ directly from the 
construction in~\ref{lem: pin by isom}.
When $k > 0$, we know that both $x_i(0)=0$ and $x_{j}(0)=0$.
For this case, we define $X_i(t) := x_i(t)/t^{k}$ and $X_{j}(t) := x_{j}(t)/t^{k}$.
Lemma \ref{lem: order} shows that 
$R_{i,j}(X_i(t),X_{j}(t))$ is well defined and analytic at 
$t=0$.  At $ t > 0$, $R_{i,j}(X_i(t),X_{j}(t))$ corresponds to the 
same geometric rotation as $R_{i,j}(x_i(t),x_{j}(t))$. 
Thus it sets $x_{j}(t)=0$ as desired.
Meanwhile at $t=0$, we have 
$x_i(0)=0$ and $x_{j}(0)=0$, so $R_{i,j}(X_i(0),X_{j}(0))$ (whatever it may be)
 keeps $x_{j}(0)=0$.
\end{proof}

\begin{lemma}
  \label{lem:transfer}
  If an unpinned trajectory  $\p(t)$ is k-trivial and $T$ is an
  isometry, then $T\p(t)$ is k-trivial.
\end{lemma}
\begin{proof}
  From the triviality assumption,
  we have $\p(t)$ in agreement
  through order $k$ with some
  $\bar{T}(t)\p(0)$, where $\bar{T}(t)$
  is an analytic trajectory of isometries with $\bar{T}(0)=I$.
  From linearity of derivative we see that $T\p(t)$ is in agreement
  through order $k$ with
  $T\bar{T}(t)\p(0)$. 
  So  $T\p(t)$ is in agreement
  through order $k$ with
  $T\bar{T}(t) T^{-1} (T\p(0))$.
 Meanwhile  $T\bar{T}(t) T^{-1}$ is an analytic trajectory of isometries that equals $I$
  at $t=0$. This means that  $T\p(t)$ is k-trivial. 

\end{proof}
We also have:
\begin{lemma}[{\cite[Lemma 4.2.3]{pss}}]
  \label{lem:ktrivTran}
  Let $T(t)$ be an analytic trajectory of isometries. Assume $T(0)=I$.
  Suppose an unpinned trajectory
  $\p(t)$ is k-trivial.
  Then $T(t)\p(t)$ is k-trivial.
\end{lemma}

Combining the two previous
lemmas, we get
\begin{lemma}
  \label{lem:ktrivTran2}
  Let $T(t)$ be an analytic trajectory of isometries.
  Suppose an unpinned trajectory
  $\p(t)$ is k-trivial.
  Then $T(t)\p(t)$ is k-trivial.
\end{lemma}

\begin{lemma}
  \label{lem:trivPinZerohigher}
  Let $\p$ be lead-spanning.
  If a trajectory $\p(t)$, with $\p(0)=\p$, is pinned and
  $k$-trivial  then it is $k$-vanishing.
\end{lemma}

\begin{proof}
  Suppose that $\p(t)$ is not
  $k$-vanishing and is $k$-trivial.
  So $\p(t)$ is $j$-active for
  some $j\le k$
  Then from~\cite[Lemma 4.2.2]{pss},
  $\p+\p^{(j)}t$ 
  is $1$-trivial.
  As discussed in \cite[p. 171, inter alia]{pss},
  for a $1$-trivial flex, there is a $d\times d$
  skew-symmetric $A$ and a vector $\b\in \RR^d$ such that, for all $i$
  \ba
  \p^{(j)}_i = A \p_i + \b.
  \ea
  Because $\p(t)$ is pinned,
  the same thing must be true at any order, so in particular,
  $\p^{(j)}$ is pinned.
  That $\b = 0$ follows from the fact that $\p_1^{(j)}$ must be $0$.

  Let $\p$ have an $\ell$-dimensional affine span.
  We  show below by induction
 that the first $\ell$ columns of $A$
  are zero.
  Then, since $\p$ is pinned and lead-spanning,
  all of the $\p_i$ are supported only
  in the first $\ell$ coordinates.
  Thus we will necessarily have $\p_i^{(j)}=0$
  for all $i$.
  This will contradict the assumed $j$-activity, yielding the
  desired contradiction.

  For the base case of the induction,
  observe that due to pinning,
  $\p^{(j)}_2$ is a
  multiple of $e_1$.
  So the above discussion implies that
  $A\p_2 = \alpha e_1$ for some $\alpha$.
  Due to lead-spanning,
  the first coordinate of $\p_2$ is
  non-zero.
  So  the first column of $A$ equals $\beta e_1$
  for some $\beta$.
  Since $A$ is
  skew-symmetric, its diagonal entries are zero, from which
  we conclude that $\beta = 0$; i.e., that the first
  column (and thus also first row)
  of $A$ is zero.

  We now induct on columns $i=[2...\ell]$.
  For the inductive step, we assume
  the statement for $i-1$.  We know that $\p^{(j)}_{i+1}$
  is in the span of $e_1, \ldots, e_{i}$ and so
  $A\p_{i+1}$ is as well.  By the inductive hypothesis 
  and skew -symmetry of
  $A$, we get that
  $A\p_{i+1} = \alpha e_{i}$, since the
  first $i-1$ entries of the $(i)$th column of $A$ are known
  to be zero.
  Due to lead-spanning,
  the $i$th coordinate of $\p_{i+1}$ is
  non-zero.
  So  the $i$th column of $A$ equals $\beta e_i$
  for some $\beta$.
  Because $A$ has zeros on the diagonal, we conclude
  that $\beta = 0$, which closes the induction.
\end{proof}

\begin{proof}[Proof of Proposition \ref{prop:jkpin}]
  Let $(G,\p)$ a lead-spanning, unpinned framework with a $(j,k)$
  flex $\p(t)$;
  i.e., $\p(t)$ is $(j-1)$-trivial, $\p(t)$ is not $j$-trivial,
  and $\m(\p(t))$
  is $k$-vanishing.  By Lemma \ref{lem:isoTraj}, there is an analytic
  trajectory of isometries $T(t)$ such that $\q(t) := T(t)\p(t)$
  is an analytic pinned trajectory.  From Lemma
  \ref{lem:ktrivTran2} (applied to
  $T(t)$), we conclude that $\q(t)$ is $(j-1)$-trivial.  Applying
  Lemma \ref{lem:ktrivTran2} to $T^{-1}(t)$ (which is also analytic),
  we get that $\q(t)$ is not $j$-trivial.
  By Lemma \ref{lem:trivPinZerohigher}, because
  $\q(0)$ is lead spanning,
  $\q(t)$ is $(j-1)$-trivial
  and $\q(t)$ is
  pinned, $\q(t)$ is $(j-1)$-vanishing.
  Clearly,because $\q(t)$ is not $j$-trivial, it
  is not $j$-vanishing. Thus it is $j$-active.
  Because $\m(\p(t)) = \m(T\p(t)) = \m(\q(t))$,
  the derivatives are the same at all orders, so we get that
  $\m(\q(t))$ is $k$-vanishing.
  In summary, $\q(t)$ is a pinned $(j,k)$ flex.

  For the other direction, let $\q$ be a lead-spanning, pinned
  configuration.
  Let $\q(t)$ be a pinned $(j,k)$ flex.
  Let $(G,\p)$ be a framework that is congruent to $\q$.  There
  is an isometry $T$ so that $\p = T\q$.
  We define $\p(t) := T\q(t)$.
  This is an analytic trajectory starting at $\p$.  We will show that
  $\p(t)$ is an unpinned $(j,k)$ flex.
  As before, since $\m$ is invariant to
  isometries, $\m(\p(t))$ is $k$-vanishing because $\m(\q(t))$
  is.  Because
  $T$ is affine and $\q(t)$ is $(j-1)$-vanishing, $\p(t)$ is
  also $(j-1)$-vanishing,
  which, in particular implies it is $(j-1)$-trivial.  By Lemma
  \ref{lem:transfer},
  if $\p(t)$ was $j$-trivial, then $\q(t)$ would also be
  $j$-trivial and then, by
  Lemma \ref{lem:trivPinZerohigher} $j$-vanishing.  Since $\q(t)$ is
  not $j$-vanishing,
  we reach a contradiction and
  thus  conclude that $\p(t)$ is not $j$-trivial.


\end{proof}

\begin{remark}
When $\p$ is lead-spanning and
has a deficient, $\ell$-dimensional, span,
we can also just apply the pinning construction 
to the first $\ell+1$ vertices $\p$ and of nearby configurations $\q$. We call such configurations
\defn{$\ell$-pinned}.  Under this notion of pinning,
Propositions~\ref{prop: pinned rigid} and~\ref{prop:jkpin}  still hold, and 
with the same proofs.
\end{remark}

\newpage
\section{Coordinates}\label{sec:coords}

\paragraph{Half flat prism, Figure \ref{fig:k33} (left).}

\phantom{.}

\noindent Coordinates:\\
\begin{tabular}{c c}
  -1   &  1\\
  -1 & 1 \\
  1 & 1\\
  1 & -1\\
  -1 & 0\\
  2 & 0\\
\end{tabular}

\noindent Edges:\\
\begin{tabular}{cc}
  (1,2),(2,4),(1,3),(3,4),(1,5),(2,5),(3,6),(4,6),(5,6)
\end{tabular}

\paragraph{Leonardo-3, Figure \ref{fig:k33} (right).}

\phantom{.}

\noindent Coordinates:\\
\begin{tabular}{c c}
  0&0\\
  1&-2\\
  2&0\\
  1&0\\
  1 &-1\\
  2&-1\\
  3&-1\\
\end{tabular}

\noindent Edges:\\
\begin{tabular}{cc}
  (1,4),(4,3),(4,5),(5,2),(1,2),(2,3),(3,1),(5,6),(6,7),(7,3),(7,2)
\end{tabular}

\paragraph{Flipped Prism, Figure \ref{fig:xnan2} (left).}

\phantom{.}

\noindent Coordinates:\\
\begin{tabular}{c c}
  -1& -1 \\
  -1 &1\\
  1& 1\\
  1 &-1\\
  -2& 0\\
  2 & 0\\
\end{tabular}\\

\noindent Edges: Same as half flat prism.

\paragraph{Asymmetric Flipped Prism, Figure \ref{fig:xnan2} (middle).}

\phantom{.}

\noindent Coordinates:\\
\begin{tabular}{c c}
  -1.183215956619924  &-1\\
  -1.183215956619924 &  1\\
  2.366431913239848  & 2\\
  1.183215956619924  &-1\\
  -2.366431913239848   &                0\\
  2.366431913239848   &                0
\end{tabular}\medskip

\noindent Edges: Same as half flat prism.

\paragraph{K33, Figure \ref{fig:xnan2} (right).}

\phantom{.}

\noindent Coordinates:\\
\begin{tabular}{c c}
  -1.829040214974252  &-0.599547897288109\\
  0.324472983863170  & 0.989869431354628\\
  -2.262180494279219  & 0.141980663717513\\
  1.370166120571291 & -0.800338877512150\\
  0.898912230425408  & 0.919393148593120\\
  -2.269078350996699 & -0.119054271080971\\
\end{tabular}

\noindent Edges:\\
\begin{tabular}{cc}
  (1,4),(1,5),(1,6),(2,4),(2,5),(2,6),(3,4),(3,5),(3,6)
\end{tabular}

\paragraph{Sphere packing 1, Figure \ref{fig:sph2} (left)}

\phantom{.}

\noindent Coordinates:\\
\begin{tabular}{c c c}
  0.555555555555555 &  1.283000598199169 & -0.907218423253029\\
  0           &        0            &       0\\
  -0.048895814363237 & -0.090211419313511 & -1.688450701386537\\
  1.481861172546362 & -0.744680624721761 & -1.676619470036412\\
  1&  0  &0\\
  1.333333333333333 &  0.769800358919501 & -0.544331053951817\\
  0.555554707548275 &  0.705650373518420 & -1.723715035652841\\
  0.000000693824970 & -0.577350669769242 & -0.816496297674721\\
  0.5 &  0.866025403784439 &  0\\
  0.833332928602402 & -0.152684581105811 & -2.155108551573311\\
  0.500000520368547 & -0.922485207078112 & -1.610777308786648\\
  1.333333102058343&   0.192450757362848 & -1.360828106966903\\
  1.000000693824489&  -0.577349868609453&  -0.816496864180240\\
  0.5&   0.288675134594813&  -0.816496580927726\\
\end{tabular}

\noindent Edges:\\
\begin{tabular}{ll}
  (1,6),(1,7),(1,9),(1,14),(2,5),(2,8),(2,9),(2,14),(3,7),(3,8),(3,10),(3,11),(4,10),\\(4,11),
  (4,12),(4,13),(5,6),(5,9),(5,13),(5,14),(6,9),(6,12),(6,14),(7,10),(7,12),\\(7,14),(8,11),
  (8,13),(8,14),(9,14),(10,11),(10,12),(11,13),(12,13),(12,14),(13,14)
\end{tabular}

\paragraph{Sphere packing 2, Figure \ref{fig:sph2} (middle)}

\phantom{.}

\noindent Coordinates:\\
\begin{tabular}{c c c}
  0.555555555555556  & 1.283000598199168  & 0.907218423253029\\
  0&   0 &  0\\
  1.068487666406798  & 0.087534503658474  & 2.319394008434100\\
  0.5 & -1.225331525265314  & 1.391057536592635\\
  1 &0   & 0\\
  1.333333333333333   &0.769800358919501  & 0.544331053951817\\
  0.555555555555556  & 0.705650329009543  & 1.723715004180755\\
  0  & -0.577350269189626  & 0.816496580927726\\
  0.5 & 0.866025403784439 &  0\\
  1.276518925548863  & -0.719615719820751  & 1.766916396256734\\
  1.333333333333333  &  0.192450089729875  & 1.360827634879543\\
  1 &  -0.577350269189626 &   0.816496580927726\\
  0.5  & 0.288675134594813  &  0.816496580927726\\
\end{tabular}

\noindent Edges:\\
\begin{tabular}{ll}
  (1,6),(1,7),(1,9),(1,14),(2,5),(2,8),(2,9),(2,14),(3,7),(3,10),(3,11),(3,12),(4,8),\\(4,10),
  (4,11),(4,13),(5,6),(5,9),(5,13),(5,14),(6,9),(6,12),(6,14),(7,10),(7,12),\\(7,14),(8,10),(8,13),(8,14),(9,14),(10,11),(11,12),(11,13),(12,13),(12,14),(13,14)
\end{tabular}

\paragraph{Coned prism, Figure \ref{fig:sph2} (right)}

\phantom{.}

\noindent Coordinates:\\
\begin{tabular}{c c c}
  1.870828693386962   & 1 & 1\\
  1.870828693386962   & 1 & -1\\
  1.870828693386962  & -1 & 1\\
  1.870828693386962  & -1 & -1\\
  1.870828693386962  & 1 & 0\\
  -1.870828693386962  & -1 & 0\\
  0         &          0  & 2\\
\end{tabular}

\noindent Edges:\\
\begin{tabular}{ll}
  (1,2),(2,4),(4,3),(3,1),(5,1),(5,2),(6,3),(6,4),(5,6),(7,1),(7,2),(7,3),(7,4),(7,5),(7,6)
\end{tabular}

\end{document}